\documentclass[a4paper,10pt,reqno]{amsart}
\usepackage[foot]{amsaddr}
\usepackage{amsmath}
\usepackage{amsthm,enumerate}
\usepackage{mathtools}
\usepackage{hyperref,url}

\usepackage{graphicx}
\usepackage{amssymb}

\usepackage{appendix}

\usepackage[italian,english]{babel}

\usepackage{fancyhdr}

\usepackage{calc}
\usepackage{url}

\usepackage[text={6in,8.6in},centering]{geometry}

\usepackage{srcltx}

\usepackage[T1]{fontenc}

\usepackage[usenames,dvipsnames]{color}

\theoremstyle{plain}
\newtheorem{theorem}{Theorem}[section]
\theoremstyle{plain}
\newtheorem{lemma}{Lemma}[section]
\newtheorem{proposition}{Proposition}[section]
\newtheorem{corollary}{Corollary}[section]

\newtheorem{definition}{Definition}[section]
\newtheorem{remark}{Remark}[section]

\newcommand{\<}{\left\langle}
\renewcommand{\>}{\right\rangle}

\renewcommand{\]}{\right]}
\newcommand{\eps}{\varepsilon}

\newcommand{\To}{\longrightarrow}

\newcommand{\be} {\begin{equation}}
\newcommand{\ee} {\end{equation}}
\newcommand{\bea} {\begin{eqnarray}}
\newcommand{\eea} {\end{eqnarray}}
\newcommand{\Bea} {\begin{eqnarray*}}
	\newcommand{\Eea} {\end{eqnarray*}}
\newcommand{\pa} {\partial}

\newcommand{\al} {\alpha}
\newcommand{\ba} {\beta}
\newcommand{\de} {\delta}

\newcommand{\ga} {\gamma}
\newcommand{\Ga} {\Gamma}

\newcommand{\om} {\omega}
\newcommand{\De} {\Delta}
\newcommand{\la} {\lambda}

\newcommand{\nequiv} {\not\equiv}
\newcommand{\Si} {\Sigma}

\newcommand{\no} {\nonumber}
\newcommand{\noi} {\noindent}
\newcommand{\lab} {\label}

\newcommand{\R}{\mathbb R}
\newcommand{\N}{\mathbb N}
\newcommand{\Rn}{\mathbb R^N}

\newcommand{\Hs}{H^s(\mathbb{R}^{N})}
\newcommand{\Iaf}{I_{a,f}(u)}

\newcommand{\ul}{u_{locmin}(a,f;x)}
\newcommand{\Jaf}{J_{a,f}(v)}

\newcommand{\taf}{t_{a,f}(v)}
\newcommand{\fhs}{\|f\|_{H^{-s}(\Rn)}}
\newcommand{\hms}{H^{-s}(\mathbb{R}^N)}
\catcode`\@=11

\newcommand{\authorfootnotes}{\renewcommand\thefootnote{\@fnsymbol\c@footnote}}%

\def\e{{\text{e}}}
\def\N{{I\!\!N}}

\numberwithin{equation}{section} \allowdisplaybreaks

\begin{document}
        \title[Existence and multiplicity of positive solutions]{Existence and Multiplicity of positive solutions of certain nonlocal scalar field equations}

\date{}
 
\author[Mousomi Bhakta]{Mousomi Bhakta}
\address{ Department of Mathematics, Indian Institute of Science Education and Research, Dr. Homi Bhaba Road, Pune-411008, India}
\email{mousomi@iiserpune.ac.in}

\author[Souptik Chakraborty]{Souptik Chakraborty}
\email{soupchak9492@gmail.com}

\author[Debdip Ganguly]{Debdip Ganguly}
\email{debdip@iiserpune.ac.in}


\begin{abstract}

We study existence and multiplicity of positive solutions of the following class of nonlocal scalar field equations:
\begin{equation}
  \tag{$\mathcal P$}
\left\{\begin{aligned}
		(-\Delta)^s u + u &= a(x) |u|^{p-1}u+f(x)\;\;\text{in}\;\mathbb{R}^{N},\\
		%u &>0 \quad\text{in}\quad\mathbb{R}^{N},\\
		u &\in H^{s}{(\mathbb{R}^{N})}
		 \end{aligned}
  \right.
\end{equation}
where $s \in (0,1)$, $N>2s$, $1<p<2_s^*-1:=\frac{N+2s}{N-2s}$, $0< a\in L^\infty(\Rn)$ and $f\in H^{-s}(\Rn)$ is a nonnegative functional i.e., $\prescript{}{H^{-s}}{\langle}f,u{\rangle}_{H^s}\geq 0$ whenever $u$ is a nonnegative function  in $\Hs$. 
We prove existence of a positive solution when $f\equiv 0$ under certain  asymptotic behavior on the function $a.$ Moreover,
 when $a(x)\geq 1$, $a(x)\to 1$ as $|x|\to\infty$ and $\|f\|_{H^{-s}(\Rn)}$ is small enough (but $f\not\equiv 0$),  then we show that ($\mathcal P$) admits at least two positive solutions. Finally, we establish existence of three positive solutions to ($\mathcal P$),  under the condition that $a(x)\leq 1$ with $a(x)\to 1$ as $|x|\to\infty$  %$a$ satisfies certain asymptotic growth at infinity
and $\|f\|_{H^{-s}(\Rn)}$ is small enough (but $f\not\equiv 0$).
\end{abstract}

\maketitle 
\tableofcontents

\section{Introduction}
	In this article we study the existence and multiplicity of positive solutions to the following fractional elliptic problem in $\Rn$:
\begin{equation}
  \tag{$\mathcal P$}\label{MAT1}
\left\{\begin{aligned}
		(-\Delta)^s u + u &= a(x) |u|^{p-1}u+f(x)\;\;\text{in}\;\mathbb{R}^{N},\\
		u &>0 \quad\text{in}\quad\mathbb{R}^{N},\\
		u &\in H^{s}{(\mathbb{R}^{N})},
		 \end{aligned}
  \right.
\end{equation}
		 where
		 $s \in (0,1)$ is fixed parameter, $N>2s$, $1<p<2_s^*-1:=\frac{N+2s}{N-2s}$, $0< a\in L^\infty(\Rn)$ %$a(x)\to 1$ as $|x|\to\infty$ 
		 and $0\not\equiv f\in H^{-s}(\Rn)$ is a nonnegative functional i.e., $ \prescript{}{H^{-s}}{\langle}f,u{\rangle}_{H^s}\geq 0$ whenever $u\geq 0$.
Here $(-\De)^s$ denotes the  fractional Laplace operator which can be defined for the Schwartz class functions $\mathcal{S}(\Rn)$  as follows:
\begin{equation} \label{De-u}
  \left(-\Delta\right)^su(x): = c_{N,s} 
\, \text{P.V.} \int_{\Rn}\frac{u(x)-u(y)}{|x-y|^{N+2s}} \, {\rm d}y, \quad c_{N,s}= \frac{4^s\Ga(N/2+ s)}{\pi^{N/2}|\Ga(-s)|}.
\end{equation}
$$H^s(R^{N}): =\bigg\{u\in L^{2}(\R^N) \; : \; \iint_{\mathbb{R}^{2N}}\frac{|u(x)-u(y)|^2}{|x-y|^{N+2s}}\,{\rm d}x\,{\rm d}y<\infty\bigg\},$$
	with the Gagliardo norm
	$$\|u\|_{H^{s}(\R{^N})} := \Big(\int_{\R{^N}}|u|^2 \, {\rm d}x+\iint_{\mathbb{R}^{2N}} \frac{|u(x)-u(y)|^2}{|x-y|^{N+2s}}\,{\rm d}x\,{\rm d}y\Big)^{1/2}.$$
	\begin{definition}
		We say $u \in \Hs$ is a positive weak solution of \eqref{MAT1} if $u>0$ in $\Rn$ and for every $\phi\in\Hs$ we have,
		\bea
		\int \int_{\R^{2N}}\frac{(u(x) - u(y))(\phi(x)-\phi(y))}{|x-y|^{N+2s}}\,{\rm d}x\,{\rm d}y +\int_{\R^N} u\phi \, {\rm d}x &=& \int_{\R{^N}}au^p\phi \, {\rm d}x+ \prescript{}{H^{-s}}{\langle}f,\phi{\rangle}_{H^s},\no
		\eea
		where $\prescript{}{H^{-s}}{\langle}.,.{\rangle}_{H^s}$ denotes the duality bracket between $f$ and $\phi$. 
	\end{definition} 
	
\medskip	
	
In recent years, there has been a considerable interest in more general version of nonlinear scalar field equation with fractional diffusion

\begin{equation}\lab{9-16-3}
\left\{\begin{aligned}
		(-\Delta)^s u + V(x)u &= g(x,u)\;\;\text{in}\;\mathbb{R}^{N},\\
		u &>0 \quad\text{in}\quad\mathbb{R}^{N},\\
		u &\in H^{s}{(\mathbb{R}^{N})},
		 \end{aligned}
  \right.
\end{equation}
see e.g., the papers
 (\cite{AV, BhM, dMMS, FV, Frank, IPS}) and the references quoted  therein.
In the physical context, this type of equation arises in the study of standing waves for the fractional  Schr\"{o}dinger equation and fractional Klein-Gordon equation. First consider the fractional Schr\"{o}dinger equation 

$$i\frac{\pa\psi}{\pa t}+(-\De)^s\psi+(V(x)+\om)\psi=g(x,\psi),$$

where $\psi=\psi(x,t)$ is a complex valued function defined on $\Rn\times\R$.  Suppose we assume 
\be\lab{9-16-1}g(x, \rho e^{i\theta})=e^{i\theta}g(x,\rho), \quad\forall\, \rho,\, \theta\in\R, \quad x\in\Rn,\ee 
and  $g: \Rn\times\R\to\R$ and $g(x,.)$ is a continuos odd function and $g(x,0)=0.$ Then one can look 
for standing wave solutions, i.e.,   $\psi(x,t)=e^{i\om t}u(x)$, which led us to the following scalar field equation 
 \be\lab{9-16-2}(-\Delta)^s u + V(x)u = g(x,u)\;\;\text{in}\;\mathbb{R}^{N}.\ee
 
 \medskip

In context to fractional quantum mechanics, nonlinear fractional Schr\"odinger equation has been proposed by Laskin in (\cite{L1, L2}) in modelling some quantum mechanical 
phenomenon. In particular it arises in evaluating Feynman path integral from the Brownian-like to the L\'{e}vy-like quantum mechanical paths.

\medskip

One may also consider fractional nonlinear Klein-Gordon equation:
$$\psi_{tt}+(-\De)^s\psi+(V(x)+\om^2)\psi=g(x,\psi),$$
where $\psi:\Rn\times\R\to \mathbb{C}$ and $g$ satisfies \eqref{9-16-1}. Then one can look for standing wave solutions as before and 
once again this will led us to the equation of type \eqref{9-16-2}.

\medskip 

Equations of the type \eqref{9-16-2} with $s=1$ arise in various other contexts of physics, for example, the classical approximation in statistical mechanics, constructive field theory, false vacuum in cosmology, nonlinear optics, laser propagations, etc (see \cite{AD, C, F, G}). They are also known as nonlinear Euclidean scalar field equations (see \cite{BL-1, BL-2}) which has been studied extensively in the last few decades by many Mathematicians. We recall some of the works without any claim of completeness the 
papers (\cite{Bahri-Li, BL-1, BL-2,  DF, DN, W}) and the references quoted therein. Much of the interest has centered on the existence and multiplicity 
of solutions under various assumptions on the potential $V$ and the nonlinearity $g.$

\medskip

Under the stated assumptions for \eqref{MAT1}, problem \eqref{MAT1} can be considered as a perturbation problem of the following homogeneous equation:
		\begin{equation}\label{AT0.8}
		\begin{split}
		&(-\Delta)^s w + w = w^{p}\;\;\text{in}\;\mathbb{R}^{N},\\
		&w>0 \;\;\text{in}\;\mathbb{R}^{N},\\
		&w \in H^{s}{(\mathbb{R}^{N})}.
		\end{split}	
		\end{equation}
		
In the seminal paper, Frank, Lenzmann and Silvestre in \cite{Frank} proved that \eqref{AT0.8} has a unique (up to  a translation) ground state solution. Further, if $w$ is any positive solution of \eqref{AT0.8}, then $w$ is radially symmetric, strictly decreasing and $w\in H^{2s+1}(\Rn) \cap C^{\infty}(\Rn)$ and satisfies the decay property:
		\be\lab{9-7-1}\frac{C^{-1}}{1+|x|^{N+2s}}\leq w(x)\leq \frac{C}{1+|x|^{N+2s}}, \ee
		with some constant $C>0$ depending on $N,\;p,\;s$.

	 \medskip
	 
Our main question is whether positive solutions can still survive after a perturbation of type \eqref{MAT1}. This question have been studied by  several authors in the local case $s=1.$
The homogeneous case, i.e., $f(x) \equiv 0$ has been studied extensively by 
Bahri-Li \cite{Bahri-Li}, Berestycki-Lions \cite{BL-1} and Ding-Ni {\cite{DN}}. On the 
other hand for the non homogeneous case, i.e., $f(x) \nequiv 0$ we refer the 
works of Adachi-Tanaka \cite{Adachi}, Jeanjean \cite{Jeanjean} and  Zhu \cite{Z} where existence 
and multiplicity of  positive solutions were proved under some assumptions on the function $a.$ We
also refer the work of Cao-Zhou \cite{CZ} for the existence of positive solution with more 
general nonlinearities. 

\medskip
				
We separate the following two cases:

 \begin{itemize}
 \item $ {\bf (A_1)}:  \quad a(x)\in(0,1] \quad\forall\, x\in \R^N, \quad\inf_{x\in\Rn} a(x)>0, \quad \mu(\{x: a(x)\neq 1\})>0, \\
 \text{and}\quad a(x)\rightarrow 1\;\text{as}\;|x|\rightarrow\infty,$

\medskip 

\item  ${\bf (A_2)}: \quad a(x)\geq 1 \quad\forall\, x\in \R^N, \quad a\in L^\infty(\Rn), \quad \mu(\{x: a(x)\neq 1\})>0, \\
\text{and}\quad a(x)\rightarrow 1\;\text{as}\;|x|\rightarrow\infty,$
%where $\mu(X)$ denotes the Lebesgue measure of a set $X$. 
\end{itemize}
where $\mu(X)$ denotes the Lebesgue measure of a set $X$.

\medskip 

Now we state our main theorems
		
%{\color{blue} In the previous theorem, what happens to the set $\{x: a(x)> 1\}$? Don't we need to assume any decay condition for $a-1$, when $a\geq 1$??? }
		
\begin{theorem}\lab{th:ex-f}
Let  $a$ satisfy ${\bf (A_2)},$ $0\not\equiv f\in H^{-s}(\Rn)$ is a nonnegative functional and $S_1$ be defined as in \eqref{17-7-5}. Moreover, if 
$$ \|f\|_{H^{-s}(\Rn)}<C_pS_1^{\tfrac{p+1}{2(p-1)}} \quad\text{where}\quad C_p:=(p\|a\|_{L^\infty(\Rn)})^{-\frac{1}{p-1}}\big(\frac{p-1}{p}\big),$$
then \eqref{MAT1} admits at least two positive solutions.
\end{theorem}

% if either of the following is satisfied:
%or,
%
%\item $\|f\|_{H^{-s}(\Rn)}=C_pS_1^{\tfrac{p+1}{2(p-1)}}$ and $\displaystyle\int_{\Rn}fu\, dx <S_1\|f\|_{H^{-s}(\Rn)} \quad\forall\, u\in X$, where
%$$X:=\{u\in \Hs: \|u\|_{L^{p+1}(\Rn)}=1 \quad\text{and}\quad \|u\|_{\Hs}^2=S_1\}.$$
%\end{itemize}

\begin{theorem}\label{MT}
Suppose  %$a\in C(\Rn)$ 
$a$ satisfies ${\bf (A_1)}$ and

\be\label{AT0.15}
1-a(x) \leq \frac{C}{1+|x|^{\mu(N+2s)}} \quad \forall\, x\in\Rn,	
\ee

for some $\mu>p+1+\frac{N}{N+2s}$. Then there exists $\delta_{0}>0$ such that for any $ 0 \not\equiv f\in H^{-s}(\Rn)$  
with $f$ is a non-negative functional and $\|f\|_{H^{-s}(\Rn)}\leq \delta_{0}$, problem \eqref{MAT1} admits at least three positive solutions.
\end{theorem}

\begin{remark}		
For the above two Theorems, it was necessary that $\|f\|_{H^{-s}(\Rn)}$ sufficiently small but $f\not\equiv 0$. In contrast, our next existence result holds in the case when $f\equiv 0$. 
\end{remark}

\begin{theorem}\lab{th:1}
Let $f\equiv 0$, $0< a\in L^\infty(\Rn)$ and there exists $a_0>0$ such that
  $$ \lim_{|x|\to\infty} a(x)= a_0=\inf_{x \in \Rn} a(x).$$

%  
%  In addition, assume that $a$ satisfies %and there exists a $\mu>0$ such that
% \be\label{AT0.15}
%(1-a(x))_+ \leq \frac{C}{1+|x|^{\mu(N+2s)}},	
%\ee
%for some $\mu>\frac{2(N+s)}{N+2s}$ and for all  $|x|>>1$. 

Then, there exists a positive solution to \eqref{MAT1} for every  $1<p<2_s^*-1$. 		
\end{theorem}

\medskip 

Like in the local case, it is well known that the Sobolev embedding 

$$
H^s(\Rn) \hookrightarrow L^{p}(\Rn) \quad \mbox{for} \ 2 \leq p \leq \frac{2N}{N - 2s},
$$
 is not compact. Thus the variational functional associated with \eqref{MAT1} fails to satisfy the Palais-Smale (PS) condition.  The lack of compactness becomes clear when one looks at the special case \eqref{AT0.8}. Solutions of \eqref{AT0.8} are invariant under translation and therefore, it is not compact. Thus the standard variational technique can not be applied directly. The existence and multiplicity results obtained in the local case were based on the careful analysis of the Palais-Smale level. However one of the major differences in the nonlocal case $s\in(0,1)$ 
  with the local case $s=1$ is due to the difference in Palais-Smale decomposition theorem. 
  
  \medskip
  
  In the case of $s=1$, we see that Palais-Smale condition holds for $\bar I_{a,f}$ (see Section 2 for the definitions) at level $c$ if $c$ can not 
  be decomposed as $c=\bar I_{a,f}(\bar{u})+k\bar I_{1,0}(w)$, where $k\geq 1$, $\bar u$ is a solution of \eqref{MAT1} and $w$ is the unique 
  radial solution of \eqref{AT0.8} (with $s=1$). But in the case of $s\in(0,1)$, uniqueness of positive solution of \eqref{AT0.8} is not yet known, 
  only the uniqueness of ground state solution is known (\cite{Frank}). Therefore, studying the Palais-Smale decomposition theorem (see Proposition \ref{PSP}), 
  we can not exclude the possibility of breaking down of Palais-Smale condition at the level $c$ for 
$c\in \big(\bar I_{a,f}(u)+\bar I_{1,0}(w^*),\, \bar I_{a,f}(u)+ 2\bar I_{1,0}(w^*)\big)$, where $w^*$ is the unique ground state solution of \eqref{AT0.8} and $u$ is any positive solution of \eqref{MAT1}. 
Thus one can not argue using Palais-Smale decomposition to obtain positive solutions to \eqref{MAT1} whose energy level 
is strictly greater than $\bar I_{a,f}(u)+\bar I_{1,0}(w^*)$.
 For the same reason, arguments of Bahri-Li \cite{Bahri-Li} can not be adopted here to prove 
Theorem \ref{th:1} even if we assume $\lim_{|x|\to\infty}a(x)=1$.

\medskip 

It is worth mentioning about the novelty of the paper. In the local case $s = 1$,  solutions of \eqref{AT0.8} has exponential decay, where as for $s\in(0,1)$, solutions of \eqref{AT0.8} has polynomial decay of the rate $|x|^{-(N+2s)}.$ Thus it is not at all straight forward to guess that the energy estimates would stay in the desired level in the nonlocal case 
 and hence deriving such estimates require a very careful analysis. Due to this fact we are able to prove Theorem \ref{MT} under much 
 weaker growth rate assumption of $a$ at infinity (see \eqref{AT0.15}) compared to the local case $s=1$(see \cite{Adachi}), where it was assumed 
		$$1-a(x)\leq C\exp\big(-(2+\de)|x|\big)\quad\text{for all}\, x\in\Rn,$$
		for some constant $\de>0, C>0$.
		
%To prove Theorem \ref{th:1}, we first establish the decomposition of Palais-Smale sequence for the functional associated with \eqref{MAT1}. We see that Palais-Smale condition fails at the level $S_m$, $m=1,2,\cdots$. Then doing a very delicate analysis of energy estimates (see Lemma \ref{BL2.2}), we construct a min-max level $c_0$ such that $c_0\in (S_1, S_2)$. Therefore, using standard deformation lemma and maximum principle, we prove the existence of a positive solution.

\medskip 

Now let us briefly explain the methodology to obtain our results. 
To prove Theorem \ref{th:ex-f}, we first decompose $\Hs$ into three components which are homeomorphic to the interior, boundary and the exterior of the unit ball in $\Hs$ respectively.  Then using assumption $(A_2)$, we prove that the energy functional associated to \eqref{MAT1} attains its infimum on one of the components which serves as our first positive solution.  The second positive solution is obtained via a careful analysis on the (PS)-sequence associated to the energy functional and we construct a mim-max critical level $\ga$, where the (PS) condition holds.  That leads to the existence of second positive solution.

\medskip 
				
To prove Theorem \ref{MT}, we establish existence of first positive solution as a perturbation of $0$ (which actually solves the problem for $f\equiv 0$) via Mountain Pass theorem. We obtain the second and third solutions of \eqref{MAT1} using Lusternik-Schnirelman category where the main problem lies in the breaking down of Palais-Smale condition at some level $c$ and we have proved that below the level of breaking down of Palais Smale condition  there are two other critical points of the energy functional associated to \eqref{MAT1}. 
	
\medskip 
			
In order to prove Theorem \ref{th:1}, we first establish existence of a positive solution $u_k$ to the following problem: 
\begin{equation*}
\left\{\begin{aligned}
		(-\Delta)^s u + u &= a(x) |u|^{p-1}u \;\;\text{in}\;B_k,\\
		%u &>0 \quad\text{in}\quad B_k,\\
		u &=0 \quad\text{in}\quad \Rn\setminus B_k,
		 \end{aligned}
  \right.
\end{equation*}
where $B_k$ is the ball of radius $k$ centered at $0$. Then we show $\|u_k\|_{\Hs}$ is uniformly bounded and there exists $0\leq \bar u\in \Hs$ such that up to a subsequence $u_k\rightharpoonup \bar u$ in $\Hs$ and $\bar u$ is a positive solution of \eqref{MAT1}. The main difficulty in this proof lies in  showing that $\bar u$ i.e., the weak limit of the subsequence $u_k$ is a nontrivial element in $\Hs$. 

\vspace{2mm}		
			
This paper has been organised in the following way:
In Section 2, we prove the Palais-Smale decomposition theorem associated with the functional corresponding to \eqref{MAT1}. %In Section 3 , we prove existence of a positive solution to \eqref{MAT1} when $f\equiv 0$, namely Theorem \ref{th:1}.  
In Section 3, we show existence of two positive solutions of \eqref{MAT1} under the assumption $(A_2)$, namely Theorem \ref{th:ex-f}. In Section 4, we prove Theorem \ref{MT}. In section 5, we prove Theorem \ref{th:1}.

\vspace{2mm}
	
	{\bf Notation:} In this paper $C$ denotes the generic constant which may vary from line to line. We denote $u_+(x):=\max\{u(x), 0\}$ and $u_-(x):=-\min\{u(x), 0\}$. Therefore, according to our notation $u=u_+-u_-$. By $w^*$, we denote the unique ground state solution of \eqref{AT0.8}.
	%By $C_b(\Rn)$, we denote the bounded continuous functions in $\Rn$. 	
	
	\section{Palais-Smale  characterization}
	
%	First we consider \eqref{MAT1} with $f\equiv 0$, i.e., 
%	
%	\begin{eqnarray}\label{BL1}
%	%
%	(-\Delta)^s u + u &=& a(x) u^{p}\;\;\text{in}\;\mathbb{R}^{N},\no\\
%	u &>&0 \;\;\text{in}\;\mathbb{R}^{N},\no\\
%	u &\in& H^{s}{(\mathbb{R}^{N})}
%	%\end{split}	
%	\end{eqnarray}
%	
%	\medskip
%	
%	where $0<a\in C_{b}(\mathbb{R}^{N})$ with $a(x)\rightarrow 1 \;\text{as,}\;|x|\rightarrow\infty$ and $1<p<2_{s}^{*}-1.$
%	
%	Since $u>0$, the Euler-Lagrange functional corresponding to \eqref{BL1} can be defined as:

In this section we study the Palais-Smale sequences (in short, PS sequences) of the functional associated to \eqref{MAT1}.

	\bea\label{EF}
			\bar I_{a,f}(u)&=& \frac{1}{2}\iint_{\R^{2N}}\frac{|u(x) - u(y)|^2}{|x-y|^{N+2s}}\,{\rm d}x\,{\rm d}y + \frac{1}{2}\int_{\R^N}|u|^2 \,{\rm d}x - \frac{1}{p+1}\int_{\R^N}a(x)|u|^{p+1}\,{\rm d}x-\prescript{}{H^{-s}}{\langle}f,u{\rangle}_{H^s}\no\\
&=&\frac{1}{2}\|u\|^{2}_{H^s(\R^N)}-\frac{1}{p+1}\int_{\R^N}a(x)|u|^{p+1}\,{\rm d}x -\prescript{}{H^{-s}}{\langle}f,u{\rangle}_{H^s},
		\eea
where $0<a\in L^\infty(\Rn)$, $a(x)\to 1$ as $|x|\to\infty$ and $0\not\equiv f\in H^{-s}(\Rn)$ is a nonnegative functional i.e., $\prescript{}{H^{-s}}{\langle}f,u{\rangle}_{H^s}\geq 0$ whenever $u\geq 0$.

We say that the sequence $u_k\in H^s(\R^N)$ is a PS sequence for $\bar I_{a,f}$ at level $\ba$ if $\bar I_{a,f}(u_k)\to \ba$ and $(\bar I_{a,f})'(u_k)\to 0$ in $H^{-s}(\Rn)$. It is easy to see that the weak limit of a PS sequence solves \eqref{MAT1} (with $f\equiv 0$) except the positivity.
%\begin{eqnarray}\lab{8-7-1}
%	(-\Delta)^s u + u &=& a(x) |u|^{p-1}u\;\;\text{in}\;\mathbb{R}^{N},\no\\
%	u &\in& H^{s}{(\mathbb{R}^{N})}.	
%	\end{eqnarray}

However the main
difficulty is that the PS sequence may not converge strongly and hence the weak limit can be zero even if $\ba>0.$
%{\color{blue} In this case lack of compactness can occur through vanishing of the mass (in the sense of Lions)}.
 The main purpose of this section is to classify PS sequences for 
the functional $\bar I_{a,f}$. Classification of PS
sequences has been done for various problems having lack of compactness, to quote a few, we cite \cite{Bahri, BC, Lions}. We establish
 a classification theorem for the PS sequences of \eqref{EF} in the spirit of the above results.\\

{\bf Throughout this section we assume $a(x)\to 1$ as $|x|\to\infty$.}		
	
	\begin{proposition}\label{PSP}
				
%		(We can write $u= u^+ - u^-$ and then taking $u^-$ as a test function and multiplying with the energy functional associated to \eqref{MAT1} and then applying integration by parts we can see $u^-\equiv 0$ and so it is sufficient to take \eqref{EF} as the energy functional associatd to \eqref{BL1} in order to find positive solutions)\\
%		(or, we can say by maximum prnciple)\\
Let $\{u_k\}\subset H^s(\R^N)$ be a PS sequence for $\bar I_{a,f}$. Then there exists a subsequence (still denoted by $u_k$) for which the following hold : \\
		there exists an integer $m\geq 0$, sequences $x_{k}^{i}$ for $1\leq i \leq m$, functions $\bar{u},\;w_{i}$ for $1\leq i\leq m$ such that
		\be
		(-\De)^s\bar{u} + \bar{u} = a(x){|\bar u|}^{p-1}\bar{u}+f \quad\text{in}\quad\R^N
		\ee
		
		\be
		\begin{split}
			(-\De)^s w_i + w_i = w_{i}^p \;\text{in}\;\R^N\\
			w_i \in H^s(\R^N), \;w_i \nequiv 0
		\end{split}
		\ee
		
		\be
		\begin{split}
			u_k-\big(\bar{u} +\sum_{i=1}^{m}w_i(\bullet-x_{k}^{i})\big) \rightarrow 0 \;\text{as}\;k\rightarrow\infty\\
			\bar I_{a,f}(u_k)\rightarrow \bar I_{a,f}(\bar{u})+\sum_{i=1}^{m}\bar I_{1,0}(w_i)\;\text{as}\;k\rightarrow \infty
		\end{split}
		\ee
		
		\be
		|x_{k}^{i}|\rightarrow\infty, \;|x_{k}^{i}-x_{k}^{j}|\rightarrow\infty\;\text{as}\;k\rightarrow\infty,\;\text{for}\;1\leq i\neq j\leq m,
		\ee
		\noi where we agree in the case $m=0$, the above holds without $w_i, x_{k}^{i}.$
%		and 
%		\be\lab{I-infty}
%		I^{\infty}(u) = \frac{1}{2}\iint_{\R^{2N}}\frac{|u(x) - u(y)|^2}{|x-y|^{N+2s}}\,dx\,dy +\frac{1}{2}\int_{\R^N}|u|^2\,dx-\frac{1}{p+1}\int_{\R^N}|u|^{p+1} \,dx,\quad\forall\,u\in H^s(\R^N)\ee
%		
		%(corresponding to \eqref{AT0.5})\\
		In addition if $u_k\geq 0,$ then $\bar{u}\geq 0$ and $w_i\geq 0$, for all $1\leq i \leq m$. 
		\end{proposition}

%\begin{remark}\lab{r:psp}
%Proposition \ref{PSP}, we see that PS condition holds for $\bar I_{a,f}$ at level $c$ if $c$ can not be decomposed as $c=\bar I_{a,f}(\bar{u})+kI^{\infty}(w)$, where $k\geq 1$ and $w$ is the unique radial solution of \eqref{AT0.8}.
%\end{remark}

To prove the above proposition, we first need some auxiliary lemmas.

\begin{lemma}\label{L1}
	Let $t>0$ and $2\leq q<2^*_s$. If $\{w_k\}$ is a bounded sequence in $\Hs$ and if 
	$$\sup_{y\in \R^N}\int_{B(y, t)} |w_k|^q {\rm d}x\longrightarrow 0\quad \text{as}\quad k\rightarrow \infty,$$
then $w_k\rightarrow 0$ in $L^r(\R^N)$ for all $r\in(2, 2_s^*)$.
\noi In addition, if $w_k$ satisfies 
\be\lab{9-7-2}
(-\De)^s w_k + w_k -a(x)|w_k|^{p-1}w_k-f  \longrightarrow 0\quad \text{in}\quad H^{-s}(\R^N),
\ee
then $w_k\rightarrow 0$ in $\Hs.$
\end{lemma}
\begin{proof}
%For $s=1$, this lemma is due to P. L. Lions \cite{Lions} (also see \cite{Willem}). For $s\in(0,1)$ the proof follows exactly in the same way. We sketch it briefly just for reader's convenience.
Choose $\kappa\in(q, 2^*_s)$ arbitrarily. Therefore, using interpolation, we have
$$\|w_k\|_{L^\kappa({B(y,t)})}\leq \|w_k\|^{1-\la}_{L^q({B(y,t)})}\|w_k\|^{\la}_{L^{2^*_s}({B(y,t)})}\leq C\|w_k\|^{1-\la}_{L^q({B(y,t)})}\|w_k\|^{\la}_{H^s(\Rn)},$$
where $\frac{1}{\kappa}=\frac{1-\la}{q}+\frac{\la}{2^*_s}$. 
%Since $q<2^*_s$, we can choose $\kappa=2+\frac{2^*_s-2}{2^*_s}q$. This choice of $\kappa$ yields
%$\la=\frac{2}{\kappa}$. Therefore,
% 
Now, covering $\Rn$ by balls of radius $t$, in such a way that each point of $\Rn$ is contained in at most $(N+1)$ balls, we find
$$\int_{\Rn}|w_k|^{\kappa}\, {\rm d}x\leq (N+1)C^\kappa\sup_{y\in\Rn}\bigg(\int_{B(y, t)} |w_k|^q {\rm d}x\bigg)^{(1-\la)\frac{\kappa}{q}}\|w_k\|_{H^s(\Rn)}^{\la\kappa}.$$ Therefore, the hypothesis of the lemmas implies $w_k\to 0$ in $L^\kappa(\Rn)$ for all $\kappa\in(q, 2^*_s)$. This completes the lemma if $q=2$, otherwise, if $q>2$, then again one can argue in similar way by choosing $\kappa\in (2, q)$.  
In addition, if \eqref{9-7-2} is satisfied, then we obtain
\be\lab{7-8-1} |\prescript{}{H^{-s}}{\langle} (-\De)^s w_k + w_k -a(x)w_k^{p-1}w_k-f, \, w_k{\rangle}_{H^s}|  = o(1)\|w_k\|_{H^s(\Rn)},\ee
where $\prescript{}{H^{-s}}{\langle}.,.{\rangle}_{H^s}$ denotes the duality bracket between $H^{-s}(\Rn)$ and $\Hs$.  Since $\{w_k\}$ is bounded in $H^s(\Rn)$, the RHS is $o(1)$. On the other hand, for the LHS we observe that since $w_k$ is bounded in $\Hs$ and $w_k\to 0$ in $L^r(\Rn)$, for $r\in (2,2^*)$, we must have $w_k\rightharpoonup 0$ in $\Hs$ and consequently, $\prescript{}{H^{-s}}{\langle}f, w_k{\rangle}_{H^s}=o(1)$. Also, by first part, $w_k\to 0$ in $L^{p+1}(\Rn)$. Hence, \eqref{7-8-1} yields $w_k\rightarrow 0$ in $\Hs$.
\end{proof}
\begin{lemma}\label{L2}
			Let ${\phi_k}$ weakly converges to $\phi$ in $\Hs$, then we have \\
			$$a|\phi_k|^{p-1}\phi_k - a|\phi|^{p-1}\phi \longrightarrow 0 \quad \text{in}\quad H^{-s}(\R^N).$$
		\end{lemma}
		\begin{proof}
			Defining $\psi_k$ as  $\phi_k - \phi$, we see $\psi_k \rightharpoonup 0$ in $\Hs$. In particular, $\{\psi_k\}$ is bounded in $\Hs$. Thus, up to a subsequence, $\psi_k \to 0$ in $L_{loc}^{q}(\R^N) \ \mbox{for all } \ 1<q<2_{s}^*$ and $\psi_k \to 0$ a.e.. Consequently,
			$a|\phi +\psi_k|^{p-1}(\phi + \psi_k) - a|\phi|^{p-1}\phi \rightarrow 0$ a.e.. Therefore, using Vitaly's convergence theorem, it follows 
			$a|\phi +\psi_k|^{p-1}(\phi + \psi_k) - a|\phi|^{p-1}\phi   \rightarrow 0$ in $L^{\tfrac{p+1}{p}}_{loc}(\Rn)$. We also observe that
			for every $\varepsilon>0, \;\text{there exists }C_\varepsilon >0$ such that 
			\be\lab{10-7-2}\bigg|a|\phi + \psi_k|^{p-1}(\phi + \psi_k) - a|\phi|^{p-1}\phi \bigg|^{\tfrac{p+1}{p}} \leq \varepsilon |\psi_k|^{p+1} + C_{\varepsilon} |\phi|^{p+1} .\ee
Moreover, since $\psi_k \rightharpoonup 0$ in $\Hs$ implies $\psi_k$ is uniformly bounded in $L^{p+1}(\Rn)$ and the fact that $|\phi|^{p+1}\in L^1(\Rn)$, it is easy to see from \eqref{10-7-2} that given $\eps>0$, there exists $R>0$ such that 
\be\lab{10-7-3}
\int_{\Rn\setminus B(0,R)}\bigg|a|\phi + \psi_k|^{p-1}(\phi + \psi_k) - a|\phi|^{p-1}\phi \bigg|^{\tfrac{p+1}{p}}\, {\rm d}x<\eps.
\ee
As a result, $a|\phi +\psi_k|^{p-1}(\phi + \psi_k) - a|\phi|^{p-1}\phi \rightarrow 0$ in $L^{\tfrac{p+1}{p}}(\Rn)$. Since $H^s(\Rn)$ is continuously embedded in $L^{p+1}(\Rn)$, which is the dual space of $L^{\tfrac{p+1}{p}}(\Rn)$, it follows that $a|\phi +\psi_k|^{p-1}(\phi + \psi_k) - a|\phi|^{p-1}\phi \rightarrow 0$ in $H^{-s}(\Rn)$.
\end{proof}	

    \begin{lemma}\label{L3}
    	For each $c_0 \geq 0, \;\text{there exists }\delta >0$ such that if $v\in\Hs$ solves
    	\be\label{BLA5}
    	(-\De)^s v + v = |v|^{p-1} v \text{ in }\Rn,\; v\in \Hs,
    	\ee
    	 and $\|v\|_{\Hs}\leq c_0,\;\|v\|_{L^2(\R^N)}\leq \delta,$  then $v \equiv 0.$ 
    \end{lemma}

	\begin{proof}
	Taking $v$ as a test function, it follows 
		\be\lab{10-7-1} \|v\|_{\Hs}^2=\int_{\Rn}|v|^{p+1}\, {\rm d}x \leq \|v\|_{L^2(\R^N)}^{\la(p+1)}\|v\|_{L^{2^*_s}(\Rn)}^{(1-\la)(p+1)}\leq C \delta^{\la(p+1)}\|v\|_{\Hs}^{(1-\la)(p+1)},\ee 
where $\la$ is such that $\frac{1}{p+1}=\frac{\la}{2}+\frac{1-\la}{2^*_s}$. If 
$(1-\la)(p+1)\geq 2$, i.e., $p\geq 1+\frac{4s}{N}$, then \eqref{10-7-1} implies $v\equiv 0$ as we can choose $\de$ small enough. Now if $p<1+\frac{4s}{N}$, then \eqref{10-7-1} yields
$\|v\|_{H^s(\Rn)}\leq C\de^\frac{\la(p+1)}{2-(1-\la)(p+1)}$. Therefore, choosing $\de>0$ small enough, we can conclude the lemma.
\end{proof}

\vspace{2mm}

{\bf Proof of Proposition \ref{PSP}:}		
	\begin{proof} We prove this proposition in the spirit of \cite{Bahri}. We divide the proof into few steps.

\medskip
	
\underline{\bf Step 1:} Using standard arguments it follows that any PS sequence for $\bar I_{a,f}$ is bounded in $H^s(\Rn)$. More precisely, 
\begin{align*}
\lim_{k\to\infty}\bar I_{a,f}(u_k)+o(1)+o(1)\|u_k\|_{H^s(\Rn)}\geq \bar I_{a,f}(u_k) \, - \, 
\frac{1}{p+1} (\bar I_{a,f})'(u_k)u_k & =\bigg(\frac{1}{2}-\frac{1}{p+1}\bigg)\|u_k\|_{H^s(\Rn)}^{2}\\
& - \left(1- \frac{1}{p+1} \right) \prescript{}{H^{-s}}{\langle}f, u_k{\rangle}_{H^s}.
\end{align*}
Hence boundedness follows. Consequently, up to a subsequence $u_k\rightharpoonup u$ in $H^s(\Rn)$.  Moreover, as  $(\bar I_{a,f})'(u_k)v\rightarrow 0 \;\text{as}\;k\rightarrow\infty\quad\forall\,  v\in\Hs$, we have
		\be\label{B6}
		(-\De)^su_k + u_k - a(x)|u_k|^{p-1}u_k -f= \varepsilon_k\overset{k}{\rightarrow} 0\quad \text{in}\quad H^{-s}(\R^N).
		\ee

	\medskip		
				
		\underline{\bf Step 2:} 
		From \eqref{B6} we get by letting $k \rightarrow 0,$
		\be\lab{8-1-1}\iint_{\R^{2N}} \frac{(u_k(x) - u_k(y))((v(x) - v(y))}{|x-y|^{N+2s}} \, {\rm d}x\, {\rm d}y \, + \, \int_{\R^N}u_kv\,  {\rm d}x \, - \, \int_{\R^N}a(x)\, |u_k|^{p-1} u_k v \, 
		{\rm d} x \, - \, \prescript{}{H^{-s}}{\langle}f, v{\rangle}_{H^s} \, {\rightarrow} \,  0,
		\ee
	for all $v\in\Hs$.	

\vspace{2mm}

{\bf Claim 1:} Weak limit $u$ satisfies 
 $$(-\De)^su + u = a(x) \, |u|^{p-1}u\, + \, f\text{ in }\Rn,\quad u\in \Hs. $$

	Indeed,  $u_k\rightharpoonup u$ in $H^s(\Rn)$ implies, 
	
	\begin{align*}
	&\int_{\R^{2N}} \frac{(u_k(x) - u_k(y))((v(x) - v(y))}{|x-y|^{N+2s}} \, {\rm d}x\, {\rm d}y \, + \, \int_{\R^N}u_kv\,  {\rm d}x \\ \,
	& \longrightarrow \, \int_{\R^{2N}} \frac{(u(x) - u(y))((v(x) - v(y))}{|x-y|^{N+2s}} \, {\rm d}x\, {\rm d}y \, + \, \int_{\R^N}u  v\,  {\rm d}x.
	\end{align*}
	
	Further using Lemma~\ref{L2} we conclude 
	
 $$\int_{\R^N}a(x)\, |u_k|^{p-1} u_k v \,{\rm d} x \,   \longrightarrow \, \int_{\R^N}a(x)\, |u|^{p-1} u v \, {\rm d} x. $$
In view of above the claim follows. 

\vspace{2mm}

\underline{\bf Step 3:} In this step we show that $u_k-u$ is a PS sequence for $\bar I_{a,0}$ at the level \\
$\lim_{k\to\infty} \bar I_{a,f}(u_k)-\bar I_{a,f}(u)$ and $u_k-u\rightharpoonup 0$ in $\Hs$. 

\medskip 

To see this, first we observe that using Brezis-Lieb lemma, we have
 \bea\lab{17-8-1}
 &&\iint_{\R^{2N}}\dfrac{|u_k(x) -  u_k(y)|^2}{|x-y|^{N+2s}} \, {\rm d}x \, {\rm d}y -\iint_{\R^{2N}}\dfrac{|{u}(x) - {u}(y)|^2}{|x-y|^{N+2s}} \, {\rm d}x \, {\rm d}y\no\\
 &&\qquad\qquad =
\iint_{\R^{2N}}\dfrac{|(u_k-u)(x)-(u_k-u)(y)|^2}{|x-y|^{N+2s}} \, {\rm d}x \, {\rm d}y +o(1).
\eea
\bea\lab{17-8-2}
 \int_{\R^N} | u_k|^2 \, {\rm d}x -\int_{\R^N}|{u}|^2 \, {\rm d}x =\int_{\R^N}|u_k-u|^2 \, {\rm d}x+o(1).
\eea
\bea\lab{17-8-3}
 \int_{\R^N} a(x)| u_k|^{p+1} \, {\rm d}x -\int_{\R^N}a(x)|u|^{p+1} \, {\rm d}x =\int_{\R^N}a(x)|u_k-u|^{p+1} \, {\rm d}x+o(1).
\eea
Further as $u_k\rightharpoonup u$ and $f\in H^{-s}(\Rn)$, we also have 
\be\lab{17-8-4}
\prescript{}{H^{-s}}{\langle}f,u_k{\rangle}_{H^s}\To \prescript{}{H^{-s}}{\langle}f,u{\rangle}_{H^s}.
\ee
Using above, it follows that
\Bea
\bar I_{a,0}(u_k-u)&=&\frac{1}{2}(\|u_k\|^2_{\Hs}-\|u\|^2_{\Hs})-\frac{1}{p+1}\bigg(\int_{\R^N} a(x)| u_k|^{p+1}-\int_{\R^N}a(x)|u|^{p+1}\bigg) + \circ(1)\\
&\To& \lim_{k\to\infty} \bar I_{a,f}(u_k) + \prescript{}{H^{-s}}{\langle}f,u{\rangle}_{H^s}-\bar I_{a,0}(u), \quad \mbox{as} \ k \rightarrow \infty, \\
&=&\lim_{k\to\infty} \bar I_{a,f}(u_k)-\bar I_{a,f}(u).
\Eea
 Next, note that \eqref{B6} and Claim 1 implies 
$$(-\De)^s(u_k-u)+(u_k-u)-a(x)(|u_k|^{p-1}u_k- |u|^{p-1}u)=\eps_k\to 0 \quad\text{in}\quad H^{-s}(\Rn).$$
Combining this with Lemma \ref{L2}, we conclude $I_{a,0}'(u_k-u)\to 0$ in $H^{-s}(\Rn)$.
Hence Step~3 follows.

\vspace{2mm}

\underline{\bf Step 4:} Using Lemma \ref{L1} we have, either $u_k-u \rightarrow 0$ in $H^s(\Rn),$ in that case the proof is over or  there
exists  $\alpha>0,$ such that up to a subsequence 

$$
Q_k(1) : = {\rm sup}_{y\in\R^N}\int_{B(y, 1)} |u_k-u|^2 \, {\rm d}x > \alpha > 0.
$$
Therefore we can find a sequence $\{ y_k \} \subset \Rn$ such that 

\begin{equation}\label{concentration-f}
\int_{B(y_k, 1)} |u_k-u|^2 \, {\rm d}x \geq  \alpha.
\end{equation}

Let us define $ {\tilde u_k}(x):=(u_k-u)(y_k + x)$, then using translation invariance of $H^s(\Rn),$ it implies 
 ${\tilde u_k}$ is also bounded in $H^s(\Rn)$ and hence converges weakly in $H^s(\Rn)$ to 
 $\tilde{u}.$ Now we claim that $\tilde{u} \neq 0.$ Indeed Rellich compactness theorem yields  
 $H^s(B(y_k, 1))\hookrightarrow L^2(B(y_k, 1))$ compactly embedded 
 and therefore \eqref{concentration-f} concludes the claim. 
 
 \medskip 
 
 Also it follows from the fact $u_k-u \rightharpoonup 0$ in $H^s(\Rn)$ and \eqref{concentration-f} that 
 
 $$
 |y_k| \longrightarrow \infty \quad \mbox{as} \ k \rightarrow \infty.
 $$
Now define, $$v_k := \tilde{u}_k \, -  \, \tilde{u}.$$
Note that, $\tilde u_k\rightharpoonup \tilde u$ implies $v_k\rightharpoonup 0$ in $\Hs.$ Using this and Lemma \ref{L2}, in the definition of $\bar{I}'_{1,0}(v_k) $ yields
		 \be\label{S0.8}
		 \bar{I}'_{1,0}(v_k)=o(1)\text{ in } \hms,
		 \ee		 
i.e., $$
(-\Delta)^s v_k + v_k = |v_k|^{p-1}v_k \longrightarrow 0 \quad \mbox{in} \ H^{-s}(\Rn).
$$

%We claim that:  
%		\be\label{S0.6}
%		\int_{\Rn} (1-a(x))|v_k(x)|^{p+1}{\rm d}x\xrightarrow{m\to\infty} 0
%		\ee
%Indeed, as $a(x)\xrightarrow{|x|\to\infty} 1$, given $\eps>0$ there exists $R(\eps)>0$ such that $|a(x)-1|<\eps,$ for all $|x|>R$. Therefore, using this and the fact that $a\in L^\infty(\Rn)$, we obtain
%\begin{eqnarray}
%		\bigg{|}\int_{\Rn} (1-a(x))|v_k(x)|^{p+1}{\rm d}x\bigg{|}&\leq (1+\|a\|_{L^\infty(\Rn)})\displaystyle\int_{|x|\leq R}|v_k(x)|^{p+1}{\rm d}x +\eps \int_{|x|>R} |v_k(x)|^{p+1}{\rm d}x\no\\
%		&\leq o(1)\;+\;\eps C\xrightarrow{m\to\infty} 0,\no
%		\end{eqnarray}
%		where for the last inequality we have used the fact that $v_k\rightharpoonup 0$ in $\Hs$ implies $v_k$ is uniformly bounded in $L^{p+1}(\Rn)$ and Rellich compactness theorem. Therefore,
%		 \begin{eqnarray}
%		 \bar{I}_{1,0}(v_k) &=& \tfrac{1}{2}\|v_k\|_{\Hs}-\tfrac{1}{p+1}\|v_k\|_{p+1}^{p+1}\no\\
%		 &=& \tfrac{1}{2}\|v_k\|_{\Hs}-\tfrac{1}{p+1}\int_{\Rn}(1-a(x))|v_k(x)|^{p+1}{\rm d}x -\langle f,v_k\rangle\no\\
%		 &&\qquad\qquad-\tfrac{1}{p+1}\int_{\Rn}a(x)|v_k(x)|^{p+1}{\rm d}x +\langle f,v_k\rangle\no\\
%		 &=& \bar{I}_{a,f}(v_k)+o(1).\no
%		 \end{eqnarray}
%		
%
%
%
%
%
%
%
%

\underline{\bf Step 5:} In this step we show that  
\be\label{e14}
(-\De)^s\tilde{u} +\tilde{u} = |\tilde{u}|^{p-1}\tilde u \quad \mbox{in} \  \Rn, \quad \tilde u \in H^s(\Rn) .
\ee

To prove this step, it is enough to show that for arbitrarily chosen $v\in C^\infty_0(\Rn)$, the following holds:
\be\lab{e14-1}
\<\tilde u, v\>_{\Hs}=\int_{\Rn}|\tilde u|^{p-1}\tilde u v\, {\rm d}x.
\ee
To show the above, let $v\in C^\infty_0(\Rn)$ be arbitrarily chosen. Since,
$\tilde u_k\rightharpoonup \tilde u$, using Step 3, we estimate the inner product between $\tilde u$ and $v$ as follows:
\bea\lab{18-8-1}
\<\tilde u, v\>_{\Hs}&=& \lim_{k\to\infty}\<\tilde u_k, v\>_{\Hs} \no\\
&=&\lim_{k\to\infty}\bigg[\iint_{\R^{2N}}\frac{\big((u_k-u)(x+y_k)-(u_k-u)(y+y_k)\big)\big(v(x)-v(y)\big)}{|x-y|^{N+2s}} \, {\rm d}x \, {\rm d}y\no\\
&&\qquad+\int_{\Rn}(u_k-u)(x+y_k)v(x) \, {\rm d}x\bigg]\no\\
&=&\lim_{k\to\infty}\bigg[\iint_{\R^{2N}}\frac{\big((u_k-u)(x)-(u_k-u)(y)\big)\big(v(x-y_k)-v(y-y_k)\big)}{|x-y|^{N+2s}}\, {\rm d}x \, {\rm d}y \no\\
&&\qquad+\int_{\Rn}(u_k-u)(x)v(x-y_k) \, {\rm d}x \bigg]\no\\
&=&\lim_{k\to\infty} \int_{\R^N}a(x)|(u_k-u)(x)|^{p-1}(u_k-u)(x) v(x-y_k) \, 
		{\rm d} x \no\\ %- \prescript{}{H^{-s}}{\langle}f, v(x-y_k){\rangle}_{H^s}\bigg]\no\\
&=&\lim_{k\to\infty}\int_{\R^N}a(x+y_k)|\tilde u_k(x)|^{p-1}\tilde u_k(x) v(x) \, 
		{\rm d} x .%-\lim_{k\to\infty} \prescript{}{H^{-s}}{\langle}f, v(x-y_k){\rangle}_{H^s}\no
\eea

{\bf Claim 2}: $\lim_{k\to\infty}\displaystyle\int_{\R^N}a(x+y_k)|\tilde u_k(x)|^{p-1}\tilde u_k(x) v(x) \, 
		{\rm d} x=\int_{\Rn}|\tilde u|^{p-1}\tilde u v\, {\rm d}x $.

\vspace{2mm}
To prove the claim, we estimate
\Bea
&&\bigg|\int_{\R^N}a(x+y_k)|\tilde u_k(x)|^{p-1}\tilde u_k(x) v(x) \, 
		{\rm d} x-\int_{\Rn}|\tilde u|^{p-1}\tilde u v\, {\rm d}x \bigg|\\
		&\leq& \bigg|\int_{\R^N}a(x+y_k)(|\tilde u_k|^{p-1}\tilde u_k-|\tilde u|^{p-1}\tilde u)v {\rm d}x\bigg|
		+ \bigg|\int_{\R^N}\big(a(x+y_k)-1\big)|\tilde u|^{p-1}\tilde u v\, {\rm d}x\bigg|\\
&=& I^1_k+J^1_k.				
\Eea
Since $|y_k|\to\infty$, $|\tilde u|^{p-1}\tilde u v\in L^1(\Rn)$, $a\in L^\infty(\Rn)$ and $a(x)\to 1$ as $|x|\to\infty$, using dominated convergence theorem, it follows that 
\be\lab{8-8-2}
\lim_{k\to\infty}J^1_k=0.
\ee
%Further, as $L^1$ norm of $(|\tilde u_k|^{p-1}\tilde u_k-|\tilde u|^{p-1}\tilde u)v$ is uniformly bounded, for $\eps>0$, there exists $R>0$ such that 
%$$\int_{|x|>R}\Big||\tilde u_k|^{p-1}\tilde u_k-|\tilde u|^{p-1}\tilde u)v\big|dx<\frac{\eps}{\|a\|_{L^\infty}(\Rn)}.$$
%Therefore,
%$$\bigg|\int_{|x|>R}a(x+y_k)(|\tilde u_k|^{p-1}\tilde u_k-|\tilde u|^{p-1}\tilde u)v dx\bigg|<\eps.$$
On the other hand, since $v$ has compact support, using Vitaly's convergence theorem 
\be
\lim_{k\to\infty}I^1_k\leq\lim_{k\to\infty}\|a\|_{L^\infty(\Rn)}\int_{\text{supp}\,v}\big||\tilde u_k|^{p-1}\tilde u_k-|\tilde u|^{p-1}\tilde u\big||v| {\rm d}x=0.\no
\ee
Combining the above two estimates, Claim 2 holds. 
Using Claim 2, we conclude Step 5 from \eqref{18-8-1}.

\medskip

Further, by Brezis-Lieb Lemma  
\bea
 \iint_{\R^{2N}}\dfrac{|\tilde u_k(x) - \tilde u_k(y)|^2}{|x-y|^{N+2s}} \, {\rm d}x \, {\rm d}y -\iint_{\R^{2N}}\dfrac{|\tilde{u}(x) - \tilde{u}(y)|^2}{|x-y|^{N+2s}} \, {\rm d}x \, {\rm d}y 
-\iint_{\R^{2N}}\dfrac{|v_k(x)-v_k(y)|^2}{|x-y|^{N+2s}} \, {\rm d}x \, {\rm d}y  \rightarrow 0 ; \no
\eea
\bea
 \int_{\R^N} |\tilde u_k|^2 \, {\rm d}x -\int_{\R^N}|\tilde{u}|^2 \, {\rm d}x -\int_{\R^N}|v_k|^2 \, {\rm d}x \rightarrow  0\no.
\eea
as $k \rightarrow \infty$.

\medskip

In view of the above steps, if ${\tilde u_k} - {\tilde u}$ does not converge to zero in $H^{s}(\Rn),$ we can 
repeat the procedure for the Palais-Smale (PS) sequence ${\tilde u_k} - \tilde{u}$ to land in either of the two cases. If it converges to zero then we stop or else we repeat the process. But this process has to stop in finitely many steps and we obtain 
 $\tilde{u}_1,\;\tilde{u}_2,\;\ldots,\;\tilde{u}_n$ 
denotes the limit solution of \eqref{e14} obtained through the procedure, we have
	   $$\sum_{i=1}^{n}\int_{\R^N}|\tilde{u}_i|^2 {\rm d}x \leq \liminf_{k\to\infty}\int_{\R^N}|u_k-u|^2 \, {\rm d}x.$$
	   Thus $n$ can not go to infinity in view of Lemma \ref{L3}.

\end{proof}

We end this section with the definition of some functions which will be used throughout the rest of the paper. We define, 

\be\lab{17-7-4}J(u) := \frac{\|u\|_{\Hs}^2}{\Big(\displaystyle\int_{\R^N}a(x) |u(x)|^{p+1}{\rm d}x\Big)^{\tfrac{2}{p+1}}}, \quad
	J_{\infty}(u) := \frac{\|u\|_{\Hs}^2}{\Big(\displaystyle\int_{\R^N} |u(x)|^{p+1}{\rm d}x\Big)^{\tfrac{2}{p+1}}}.\ee
\be\lab{17-7-5}S_1 := \inf_{u\in\Hs\setminus \{0\}} J_{\infty}(u). 
%,\quad S_m := m^{\tfrac{p-1}{p+1}}S_1, \quad m=2, 3, 4, \cdots
	 \ee

	 From \cite{Frank}, it is known that $S_1$ is achieved by unique ground state solution $w^*$ of \eqref{AT0.8}. Further $w^*$ is
	  radially symmetric positive decreasing smooth function satisfying \eqref{9-7-1}.

\section{Proof of Theorem \ref{th:ex-f}}

In this section we prove Theorem \ref{th:ex-f}. To this aim we first establish existence of two positive critical points  in the spirit of \cite{Jeanjean} for the following functional:
	\be\label{EF-1}
			I_{a,f}(u)=\frac{1}{2}\|u\|^{2}_{H^s(\R^N)}-\frac{1}{p+1}\int_{\R^N}a(x)u_+^{p+1}\,{\rm d}x - \prescript{}{H^{-s}}{\langle}f,u{\rangle}_{H^s},
		\ee
		where $u_+(x):=\max\{u(x), 0\}$ and $u_-(x):=-\min\{u(x), 0\}$ and $f\in H^{-s}(\Rn)$ is a nonnegative functional.
		
	Clearly, if $u$ is a critical points of $I_{a,f}$, then $u$ solves 
\begin{equation}\label{MAT2}
\left\{\begin{aligned}
		(-\Delta)^s u + u &= a(x) u_+^{p}+f(x)\;\;\text{in}\;\mathbb{R}^{N},\\
		u &\in H^{s}{(\mathbb{R}^{N})}.
		 \end{aligned}
  \right.
\end{equation}	
	
\begin{remark}\lab{r:30-7-3}
If $u$ is a weak solution of \eqref{MAT2} and $f$ is a nonnegative functional, then taking $v=u_-$ as a test function in \eqref{MAT2}, we obtain 

$$-\|u_-\|^2_{\Hs} \, -\, \iint_{\R^{2N}} \frac{[u_+(y)u_-(x)+u_+(x)u_-(y)]}{|x - y|^{N + 2s}} \, {\rm d}x \, {\rm d}y
=\prescript{}{H^{-s}}{\langle}f,u_-{\rangle}_{H^s}\geq 0.$$ This in turn implies $u_-=0$, i.e., $u\geq 0$. Therefore, using maximum principle \cite[Theorem 1.2]{DPQ}, it follows that, $u$ is a positive solution to \eqref{MAT2}. Hence $u$ is a solution to \eqref{MAT1}. 
\end{remark}

To establish the existence of two critical points for $I_{a,f}$, we first need to prove some auxiliary results. Towards that, we partition $\Hs$ into three disjoint sets.
Let, $g:\Hs\to\R$ be defined by $$g(u):=\|u\|_{\Hs}^2-p ||a||_{L^\infty(\Rn)}\|u\|^{p+1}_{L^{p+1}(\Rn)}.$$
Now, we define
$$U_1:=\{u\in\Hs: u=0 \quad\text{or}\quad g(u)>0\}, \quad U_2:=\{u\in\Hs: g(u)<0\},$$
$$U:=\{u\in\Hs\setminus\{0\}: g(u)=0\}.$$

\begin{remark}\lab{r:30-7-1}
Since $p>1$, using Sobolev inequality, it is easy to see that $\|u\|_{\Hs}$ and $\|u\|_{L^{p+1}(\Rn)}$ are bounded away from $0$, for all $u\in U$.
\end{remark}

We define,
\be\lab{30-7-1}
c_0:=\inf_{U_1} {I_{a,f}}(u) \quad\text{and}\quad c_1:=\inf_{U} {I_{a,f}}(u).
\ee

\begin{remark}\lab{r:30-7-2}
 For any $t>0$, $g(tu)=t^2\|u\|_{\Hs}^2-t^{p+1}p|a|_{L^\infty(\Rn)}\|u\|^{p+1}_{L^{p+1}(\Rn)}$.  Moreover $g(0)=0$ and $t\mapsto g(tu)$ is a strictly concave function, we have for any $u\in \Hs$ with $\|u\|_{\Hs}=1$, there exists unique $t=t(u)$ such that $tu\in U.$ On the other hand, for any $u\in U$, it holds $g(tu)=(t^2-t^{p+1})\|u\|_{\Hs}^2$. This implies 
 $$tu\in U_1 \quad\text{for all}\quad t\in (0,1) \quad\text{and}\quad tu\in U_2 \quad\text{for all}\quad  t > 1.$$
\end{remark}

\begin{lemma}\lab{l:1-8-1}
Assume $C_p$ is defined as in Theorem \ref{th:ex-f}. Then there holds,
$$\frac{p-1}{p}\|u\|_{\Hs} \geq C_pS_1^{\tfrac{p+1}{2(p-1)}} \quad\forall\quad u\in U,$$
where $S_1$ is as defined in \eqref{17-7-5}.
\end{lemma}
\begin{proof}
$u\in U$ implies, $\|u\|_{L^{p+1}(\Rn)}=\frac{\|u\|_{\Hs}^\frac{2}{p+1}}{(p||a||_{L^\infty(\Rn)})^\frac{1}{p+1}}$. Therefore, combining this with the definition of $S_1$, we have 
$$\|u\|_{\Hs}\geq S_1^\frac{1}{2}\|u\|_{L^{p+1}(\Rn)} = S_1^\frac{1}{2}\frac{\|u\|_{\Hs}^\frac{2}{p+1}}{(p||a||_{L^\infty(\Rn)})^\frac{1}{p+1}} \quad\forall\, u\in U.$$ Therefore,
 for all $u\in U$, we have
 $$\|u\|_{\Hs}\geq\frac{S_1^\frac{p+1}{2(p-1)}}{(p||a||_{L^\infty(\Rn)})^\frac{1}{p-1}}=\frac{p}{p-1}C_pS_1^\frac{p+1}{2(p-1)}.$$ Hence the lemma follows. 
\end{proof}

\begin{lemma}\lab{l:30-7-1}
Assume $C_p$ is defined as in Theorem \ref{th:ex-f} and  
\be\lab{J1.3}
\inf_{u\in\Hs,\, \|u\|_{L^{p+1}(\Rn)=1}}\bigg\{C_p\|u\|_{\Hs}^\frac{2p}{p-1}-\prescript{}{H^{-s}}{\langle}f,u{\rangle}_{H^s}\bigg\}>0.
\ee 
Then $c_0<c_1$, where $c_0$ and $c_1$ are defined as in \eqref{30-7-1}.
\end{lemma}
\begin{proof}
Define, \be\lab{31-7-1}\tilde J(u):=\frac{1}{2}\|u\|_{\Hs}^2-\frac{||a||_{L^\infty(\Rn)}}{p+1}\|u\|^{p+1}_{L^{p+1}(\Rn)}-\prescript{}{H^{-s}}{\langle}f,u{\rangle}_{H^s}, \quad u\in\Hs.\ee

{\bf Step 1}: In this step we prove that there exists $\alpha>0$ such that
$$\frac{d}{dt}\tilde J(tu)|_{t=1}\geq \al \quad\forall\, u\in U.$$ 

From the definition of $\tilde J$, we have $\frac{d}{dt}\tilde J(tu)|_{t=1}=\|u\|_{\Hs}^2-|a|_{L^\infty(\Rn)}\|u\|_{L^{p+1}(\Rn)}^{p+1}-\prescript{}{H^{-s}}{\langle}f,u{\rangle}_{H^s}$. Therefore, using the definition of $U$ and the value of $C_p$,  we have for  $u\in U$
\bea\lab{30-7-4}
\frac{d}{dt}\tilde J(tu)|_{t=1}=\frac{p-1}{p}\|u\|_{\Hs}^2-\prescript{}{H^{-s}}{\langle}f,u{\rangle}_{H^s}&=& (p|a|_{L^\infty(\Rn)})^\frac{1}{p-1}C_p\|u\|_{\Hs}^2-\prescript{}{H^{-s}}{\langle}f,u{\rangle}_{H^s}\no\\
&=&\bigg(\frac{\|u\|^2_{\Hs}}{||u||_{L^{p+1}(\Rn)}^{p+1}}\bigg)^\frac{1}{p-1}C_p\|u\|_{\Hs}^2-\prescript{}{H^{-s}}{\langle}f,u{\rangle}_{H^s}\no\\
&=&C_p\frac{\|u\|^\frac{2p}{p-1}_{\Hs}}{||u||_{L^{p+1}(\Rn)}^\frac{p+1}{p-1}}-\prescript{}{H^{-s}}{\langle}f,u{\rangle}_{H^s}.
\eea
Further, $\eqref{J1.3}$ implies there  exists $d>0$ such that

		\be\lab{30-7-3}
		\inf_{u\in\Hs,\, \|u\|_{L^{p+1}(\Rn)}=1}\Big\{C_p\|u\|_{\Hs}^{\tfrac{2p}{p-1}}-\prescript{}{H^{-s}}{\langle}f,u{\rangle}_{H^s}\Big\}\geq d.\ee
		Now, 
		\Bea
		\eqref{30-7-3}&\Longleftrightarrow& C_p\frac{\|u\|_{\Hs}^{\tfrac{2p}{(p-1)}}}{\|u\|_{L^{p+1}(\Rn)}^{\tfrac{p+1}{p-1}}}-\prescript{}{H^{-s}}{\langle}f,u{\rangle}_{H^s}\geq d, \quad \|u\|_{L^{p+1}(\Rn)}=1\\
		&\Longleftrightarrow& C_p\frac{\|u\|_{\Hs}^{\tfrac{2p}{(p-1)}}}{\|u\|_{L^{p+1}(\Rn)}^{\tfrac{p+1}{p-1}}}-\prescript{}{H^{-s}}{\langle}f,u{\rangle}_{H^s} \geq d\|u\|_{L^{p+1}(\Rn)}, \quad u\in\Hs\setminus\{0\}.
\Eea
		Hence, plugging back the above estimate into \eqref{30-7-4} and using Remark \eqref{r:30-7-1} we complete the proof of Step 1. 
		
		\vspace{2mm}
		
{\bf Step 2:} Let $u_n$ be a minimizing sequence for $I_{a,f}$ on $U$, i.e., $I_{a,f}(u_n)\to c_1$ and $\|u_n\|_{\Hs}^2=p\|a\|_{L^\infty(\Rn)}\|u_n\|_{L^{p+1}(\Rn)}^{p+1}$. Therefore, for large $n$
$$c_1+o(1)\geq I_{a,f}(u_n)\geq \tilde J(u_n)\geq \bigg(\frac{1}{2}-\frac{1}{p(p+1)}\bigg)\|u_n\|^2_{\Hs}- \|f\|_{H^{-s}(\Rn)}\|u_n\|_{\Hs}.$$
This implies that $\{\tilde J(u_n)\}$ is a bounded sequence and $\|u_n\|_{\Hs}$ and 
$\|u_n\|_{L^{p+1}(\Rn)}$ are bounded. 

{\bf Claim}: $c_0<0$. 

Indeed, to prove the claim, it's enough to show that there exists $v\in U_1$ such that $I_{a,f}(v)<0$. Note that, thanks to Remark \ref{r:30-7-2}, we can choose $u\in U$ such that $\prescript{}{H^{-s}}{\langle}f,u{\rangle}_{H^s}>0$. Therefore,
$$I_{a,f}(tu)\leq t^2|u|_{L^{p+1}(\Rn)}^{p+1}\bigg[\frac{p\|a\|_{L^\infty(\Rn)}}{2}-\frac{t^{p-1}}{p+1}\bigg]-t\prescript{}{H^{-s}}{\langle}f,u{\rangle}_{H^s} <0.$$
for $t<<1$. Also by Remark \ref{r:30-7-2}, $tu\in U_1$. Hence the claim follows.

Thanks to the above claim, $I_{a,f}(u_n)< 0$ for large $n$. Consequently,
$$0>I_{a,f}(u_n)\geq \bigg(\frac{1}{2}-\frac{1}{p(p+1)}\bigg)\|u_n\|^2_{\Hs}-\prescript{}{H^{-s}}{\langle}f,u_n{\rangle}_{H^s}.$$
This in turn implies $\prescript{}{H^{-s}}{\langle}f,u_n{\rangle}_{H^s}> 0$ for all large $n$ (since $p>1$).
Consequently, $\frac{d}{dt}\tilde J(tu_n)<0$ for $t>0$ small enough. Thus, by Step 1, there exists $t_n\in (0,1)$ such that $\frac{d}{dt}\tilde J(t_n u_n)=0$.
Moreover, $t_n$ is unique since,
	$$\frac{d^2}{dt^2}\tilde J(tu) = \|u\|_{\Hs}^2-p\|a\|_{L^\infty(\Rn)}t^{p-1}\|u\|_{L^{p+1}(\Rn)}^{p+1}= (1-t^{p-1})\|u\|_{\Hs}^2>0,\;\forall u\in U,\;\forall t\in [0,\;1).$$
	
{\bf Step 3}: In this step we show  that 	
\be\label{J3}
	\liminf_{n\rightarrow\infty}\{\tilde J(u_n)-\tilde J(t_nu_n)\}>0.
	\ee
	
We observe that, $\tilde J(u_n)-\tilde J(t_nu_n)=\displaystyle\int_{t_n}^{1}\frac{d}{dt}\{\tilde  J(tu_n)\} \, {\rm d}t$ and that for all $n\in\mathbb{N},$ there is $\xi_n>0$ such that
	$t_n\in(0,\;1-2\xi_n)$ and $\frac{d}{dt}\tilde J(tu_n)\geq \al$ for $t\in[1-\xi_n,\;1]$.\\
	To establish \eqref{J3}, it is enough to show that $\xi_n>0$ can be chosen independent of $n\in\mathbb{N}$. But this is true since, $\frac{d}{dt}\tilde J(tu_n)|_{t=1}\geq \al$ and for the boundedness of $\{u_n\},$\\
	$$\bigg|\frac{d^2}{dt^2}\tilde J(tu_n)\bigg| = \bigg|\|u_n\|_{\Hs}^2 - p\|a\|_{L^\infty(\Rn)}t^{p-1}\|u_n\|_{L^{p+1}(\Rn)}^{p+1}\bigg|= \bigg|(1-t^{p-1})\|u_n\|_{\Hs}^2\bigg|\leq C,$$
	for all $n\geq 1$ and $t\in[0,\;1]$. 
	
\vspace{2mm}	
	
{\bf Step 4:} From the definition of $I_{a,f}$ and $\tilde J$, it immediately follows that $\frac{d}{dt}I_{a,f}(tu)\geq \frac{d}{dt}\tilde J(tu)$ for all $u\in\Hs$ and for all $t>0$. Hence,
$$I_{a,f}(u_n)-I_{a,f}(t_nu_n)= \int_{t_n}^{1}\frac{d}{dt}(I_{a,f}(tu_n))\; {\rm d}t\geq \int_{t_n}^{1}\frac{d}{dt}\tilde J(tu_n)\; {\rm d}t = \tilde J(u_n)-\tilde J(t_nu_n)$$
	Since, $\{u_n\}\in U$  is a minimizing sequence for $I_{a,f}$, and $t_nu_n\in U_1,$ we conclude using \eqref{J3} that\\
	$$c_0 = \inf_{u\in U_1}I_{a,f}(u)< \inf_{u\in U}I_{a,f}(u)\equiv c_1$$
\end{proof}

Next, we introduce the problem at infinity associated to \eqref{MAT2}:
\be\lab{30-7-5}
(-\De)^s u+u=u_+^p \quad\text{in}\quad\Rn,
\ee
and the corresponding functional $I_{1,0}:\Hs\to\R$ defined by
$$I_{1,0}(u)=\frac{1}{2}\|u\|_{\Hs}^2-\frac{1}{p+1}\int_{\Rn}u_+^{p+1}\, {\rm d}x.$$
Define,
\be\lab{1-8-5}X_1:=\{u\in\Hs\setminus\{0\}: (I_{1,0})'(u)=0\}, \quad S^\infty:=\inf_{X_1} I_{1,0}.\ee
\begin{remark}\lab{r:31-7-1}
Clearly $I_{1,0}(u)=\frac{p-1}{2(p+1)}\|u\|^2_{\Hs}$ on $X_1$. From \eqref{17-7-5}, we also have $\|u\|^2_{\Hs} \geq S_1^\frac{p+1}{p-1}$ on $X_1$. Therefore, $S^\infty\geq \frac{p-1}{2(p+1)}S_1^\frac{p+1}{p-1}>0$.
Further, it's known from \cite{Frank} that $S_1$ is achieved by unique positive radial ground state solution $w^*$ of \eqref{AT0.8}. Therefore,
$$I_{1,0}(w^*)=\frac{p-1}{2(p+1)}S_1^\frac{p+1}{p-1}.$$
Hence $S^\infty$ is achieved by $w^*$.
\end{remark}

%\begin{proposition}\label{PSP-1}
%
%Let $\{u_k\}\in H^s(\R^N)$ be a bounded PS sequence for $I_{a,f}$. Then there exists a subsequence (still denoted by $u_k$) for which the following hold : \\
%		there exists an integer $m\geq 0$, sequences $\{x_{k}^{i}\}$ for $1\leq i \leq m$, functions $\bar{u},\;w_{i}$ for $1\leq i\leq m$ such that $\bar u$ satisfies \eqref{MAT2} and $w_i$ satisfy \eqref{30-7-5} and 
%\be
%\begin{split}
%u_k-\big(\bar{u} +\sum_{i=1}^{m}w_i(\bullet-x_{k}^{i})\big) \rightarrow 0 \;\text{as}\;k\rightarrow\infty\no\\
%			I_{a,f}(u_k)\rightarrow I_{a,f}(\bar{u})+\sum_{i=1}^{m}I_{1,0}(w_i)\;\text{as}\;k\rightarrow \infty \no
%		\end{split}
%		\ee
%		
%		\be
%		|x_{k}^{i}|\rightarrow\infty, \;|x_{k}^{i}-x_{k}^{j}|\rightarrow\infty\;\text{as}\;k\rightarrow\infty,\;\text{for}\;1\leq i\neq j\leq m \no.
%		\ee
%\end{proposition}
%
%The above proposition can be proved in the spirit of \cite{BC} and Proposition \ref{PSP}. We skip the details.

\begin{proposition}\lab{p:30-7-1}
Assume \eqref{J1.3} holds. Then $I_{a,f}$ has a critical point $u_0\in U_1$ with $I_{a,f}(u_0)=c_0$. In particular, $u_0$ is a positive weak solution to \eqref{MAT1}.
\end{proposition}
\begin{proof} We decompose the proof into few steps.

\vspace{2mm}

{\bf Step 1}: $c_0>-\infty$.

Since $I_{a,f}(u)\geq \tilde J(u)$, where $\tilde J$ is defined as in \eqref{31-7-1}, in order to prove Step 1, it is enough to show that $\tilde J$ is bounded from below. From definition of $U_1$, it immediately follows that
\be\lab{31-7-3}\tilde J(u)\geq [\tfrac{1}{2}-\tfrac{1}{p(p+1)}]\|u\|_{\Hs}^2 - \|f\|_{H^{-s}(\Rn)}\|u\|_{\Hs}\;\, \mbox{for all}\, u\in {U}_1.\ee
As RHS is quadratic function in $\|u\|_{\Hs}$, $\tilde J$ is bounded from below. Hence Step 1 follows.

\vspace{2mm} 

{\bf Step 2}: In this step we show that there exists a bounded PS sequence 
$\{u_n\}\subset U_1$ for $I_{a,f}$ at level $c_0$.

Let $\{u_n\}\subset \bar U_1$ such that $I_{a,f}(u_n)\to c_0$. Since $I_{a,f}(u)\geq \tilde J(u)$ from \eqref{31-7-3}, it follows that $\{u_n\}$ is a bounded sequence. Since by Lemma \ref{l:30-7-1}, $c_0<c_1$, without restriction we can assume $u_n\in U_1$. Therefore, by Ekeland's variational principle from $\{u_n\}$, we can extract a PS sequence in $U_1$ for $I_{a,f}$ at level $c_0$. We again call it by $\{u_n\}$. That completes the proof of Step 2.

\vspace{2mm}

{\bf Step 3:}  In this step we show that there exists $u_0\in U_1$ such that $u_n\to u_0$ in $\Hs$.

Applying Proposition \ref{PSP}, it follows
		\be\label{J5}
		u_n-u_0 - \sum_{i=1}^{m}w^i(x-x_n^i)\rightarrow 0\;\text{ in }\Hs
		\ee
		for some $u_0$ with $(I_{a,f})'(u_0) =0$ and some appropiate $w^i,\;\{x_n^i\}$. To prove Step 3, we need to show that $m=0$. We argue by method of contradiction. 
		Suppose there is $w^i\neq 0$ ($i\in\{1, 2,\cdots, m\}$) such that $(I_{1,0})'(w^i) = 0$. i.e, $\|w^i\|^2_{\Hs} =\int_{\Rn}(w^i_+)^{p+1}\, {\rm d}x.$ Therefore,
\Bea
g(w^i)=\|w^i\|_{\Hs}^2-p\|a\|_{L^\infty(\Rn)}\|w^i\|^{p+1}_{L^{p+1}(\Rn)} &= &\int_{\Rn}(w^i_+)^{p+1}\, {\rm d}x-p\|a\|_{L^\infty(\Rn)}\int_{\Rn}|w^i|^{p+1}\, {\rm d}x\\
&\leq&\|w^i\|^{p+1}_{L^{p+1}(\Rn)}(1-p\|a\|_{L^\infty(\Rn)})<0.
\Eea		
The last inequality follows from the fact that $p>1$ and $\|a\|_{L^\infty(\Rn)}\geq 1$. Now from Remark \ref{r:31-7-1}, 	$I_{1,0}(w^i)\geq S^\infty>0$	for all $1\leq i\leq m$. Therefore, $I_{a,f}(u_n)\rightarrow I_{a,f}(u_0)+\sum_{i=1}^{m}I_{1,0}(w_i)$ implies $I_{a,f}(u_0)<c_0$. This in turn implies, $u_0\not\in U_1$. Therefore, $g(u_0)\leq 0$.

Now we evaluate, $g\big(u_0+\sum_{i=1}^{m}w^i(x-x_n^i)\big)$. Since $u_n\in U_1$, we have $g(u_n)\geq 0$. Therefore, applying uniform continuity of $g$, we obtain from \eqref{J5}  that 
		\be\label{J8}
		0\leq \liminf_{n\rightarrow\infty}g(u_n)=\liminf_{n\rightarrow\infty} g\big(u_0+\sum_{i=1}^{m}w^i(x-x_n^i)\big).
		\ee
		On the other hand, since $|x_n^i|\rightarrow\infty,\;|x_n^i-x_n^j|\rightarrow\infty \text{, for }1\leq i\neq j\leq m$\\
		the supports of $u_0(\bullet)$ and $w^i(\bullet-x_n^i)$ are going increasingly far away as $n\to\infty$ and we get
		$$\lim_{n\rightarrow\infty}g\big(u_0+\sum_{i=1}^{m}w^i(x-x_n^i)\big) = g(u_0)+\lim_{n\rightarrow\infty}\sum_{i=1}^{m}g\bigg(w^i(x-x_n^i)\bigg)=g(u_0)+\sum_{i=1}^{m}g(w^i),$$
where the last equality is due to the fact that $g$ is invariant under translation in $\Rn$. Now since $g(u_0)\leq 0$ and $g(w^i)<0,$ for $i\leq i\neq j\leq m$, we get a contradiction to $\eqref{J8}$. Hence Step 3 follows.
		
\vspace{2mm}

{\bf Step 4:} From the previous steps we conclude that $I_{a,f}(u_0)=c_0$ and 
$(I_{a,f})'(u_0)=0$. Therefore, $u_0$ is a weak solution to \eqref{MAT2}. Combining this with Remark \ref{r:30-7-3}, we conclude the proof of the proposition. 
\end{proof}

\begin{proposition}\lab{p:31-7-2} 
Assume \eqref{J1.3} holds. Then $I_{a,f}$ has a second critical point $v_0\neq u_0$. In particular, $v_0$ is a positive solution to \eqref{MAT1}.
\end{proposition}

\begin{proof}
Let $u_0$ be the critical point obtained in Proposition \ref{p:30-7-1} and $w^*$ be as in Remark \ref{r:31-7-1}. Set, $w_t(x):=w^*\big(\frac{x}{t}\big)$

{\bf Claim 1:} $u_0+w_t\in U_2$ for $t>0$ large enough.

Indeed, as $p>1$ and $|a|_{L^\infty(\Rn)}\geq 1$, 
\Bea
g(u_0+w_t)&\leq&\|u_0\|^2_{\Hs}+\|w_t\|^2_{\Hs}+2\<u_0,w_t\>_{H^s}-p(|u_0|^{p+1}_{L^{p+1}(\Rn)}+|w_t|^{p+1}_{L^{p+1}(\Rn)})\\
&\leq&(1+\eps)\|w_t\|^2_{\Hs}+(1+C(\eps))\|u_0\|^2_{\Hs}-p(|u_0|^{p+1}_{L^{p+1}(\Rn)}+|w_t|^{p+1}_{L^{p+1}(\Rn)}),
\Eea
where to get the last inequality, we have used Young's inequality with $\eps>0$.  Further, as $w^*$ solves \eqref{AT0.8}, we have
$$\|w_t\|^2_{\Hs}=t^{N-2s}[w^*]_{H^s}^2+t^N|w^*|_{L^2(\Rn)}^2 \quad\text{and}\quad |w_t|^{p+1}_{L^{p+1}(\Rn)}=t^N\|w^*\|_{\Hs}^2,$$
where $[.]_{H^s}$ denotes the seminorm in $\Hs$. Therefore,
\Bea
g(u_0+w_t)&\leq&(1+C(\eps))\|u_0\|^2_{\Hs}-p|u_0|^{p+1}_{L^{p+1}(\Rn)}+[w^*]_{H^s}^2\big[(1+\eps)t^{N-2s}-pt^N\big]\\
&&\qquad\qquad+ \, t^N|w^*|_{L^2(\Rn)}^2\big[(1+\eps)-p\big].
\Eea
Choose $\eps>0$ such that $1+\eps<p$. Therefore, $g(u_0+w_t)<0$ for $t$ to be large enough. Hence the claim follows. 

\vspace{2mm}

{\bf Claim 2:} $I_{a,f}(u_0+w_t)<I_{a,f}(u_0)+I_{1,0}(w_t),\;\forall\, t>0 $.

Indeed, since $u_0, \, w_t>0$, taking $w_t$ as the test function for \eqref{MAT2} yields $$ {\langle} u_0,\;w_t{\rangle}_{\Hs} =\int_{\Rn}a(x)u_0^pw_t\, {\rm d}x+\prescript{}{H^{-s}}{\langle}f,w_t{\rangle}_{H^s}.$$
Therefore, using the above expression and the fact that $a\geq 1$, we obtain
\Bea
I_{a,f}(u_0+w_t)&=& \frac{1}{2}\|u_0\|_{\Hs}^2+\frac{1}{2}\|w_t\|_{\Hs}^2+ {\langle} u_0,\;w_t{\rangle}_{\Hs} \\
&&\qquad-\frac{1}{p+1}\int_{\Rn}a(x)(u_0+w_t)^{p+1}\;{\rm d}x-\prescript{}{H^{-s}}{\langle}f,u_0{\rangle}_{H^s}-\prescript{}{H^{-s}}{\langle}f,w_t{\rangle}_{H^s}\\
&=&I_{a,f}(u_0)+I_{1,0}(w_t) +\langle u_0,\;w_t\rangle_{\Hs} 
+\frac{1}{p+1}\int_{\Rn}a(x)u_0^{p+1}\;{\rm d}x\\
&&\quad+\frac{1}{p+1}\int_{\Rn}w_t^{p+1}\;{\rm d}x
-\frac{1}{p+1}\int_{\Rn}a(x)(u_0+w_t)^{p+1}\;{\rm d}x-\prescript{}{H^{-s}}{\langle}f,w_t{\rangle}_{H^s}\\
&\leq&I_{a,f}(u_0)+I_{1,0}(w_t) +\frac{1}{p+1}\int_{\Rn}a(x)\bigg[(p+1)u_0^pw_t+u_0^{p+1}+w_t^{p+1}-(u_0+w_t)^{p+1}\bigg]{\rm d}x\\
&<&I_{a,f}(u_0)+I_{1,0}(w_t). 
\Eea
Hence the Claim follows.

\vspace{2mm}

Also, by direct computation, it follows 
\be\lab{1-8-2}
I_{1,0}(w_t)=\frac{t^{N-2s}}{2}[w^*]_{H^s}^2+\frac{t^N}{2} \|w^*\|_{L^2(\Rn)}^2-\frac{t^N}{p+1}|w^*|_{L^{p+1}(\Rn)}^{p+1}\to -\infty \quad\text{as}\quad t\to\infty,
\ee
From \eqref{1-8-2}, it is also easy to see that 
$$
\sup_{t>0} I_{1,0}(w_t)=I_{1,0}(w_1)=I_{1,0}(w^*)=S^\infty,
$$
where the last equality is due to Remark \ref{r:31-7-1}.
Combing this with Claim 2 yields
\be\lab{1-8-6}
I_{a,f}(u_0+w_t)<I_{a,f}(u_0)+ S^\infty \quad\forall\, t>0.
\ee

Combining \eqref{1-8-2} with Claim 2, we have \be\lab{1-8-3}I_{a,f}(u_0+w_t)<I_{a,f}(u_0) \quad\text{for}\quad t \quad\text{large enough}.\ee
 Fix $t_0>0$ large enough such that \eqref{1-8-3} and Claim 1 are satisfied.
 
 Then we set 
$$\ga:=\inf_{i\in\Ga}\max_{t\in[0,1]} I_{a,f}\big(i(t)\big),$$
where $$\Ga:=\{i\in C\big([0,1], \Hs\big) : i(0)=u_0,\quad i(1)= u_0+w_{t_0}\}.$$
As $u_0\in U_1$ and $u_0+w_{t_0}\in U_2$, for every $i\in \Ga$, there exists $t_i\in(0,1)$
such that $i(t_i)\in U$. Therefore, $$\max_{t\in[0,1]} I_{a,f}(i(t))\geq I_{a,f}\big(i(t_i)\big)\geq \inf_{U}I_{a,f}(u)=c_1.$$
Thus, $\ga\geq c_1>c_0=I_{a,f}(u_0)$.
Here in the last inequality we have used Lemma \ref{l:30-7-1}.

\vspace{2mm}

{\bf Claim 3:} $\ga<S^\infty$, where $S^\infty$ is as defined in \eqref{1-8-5}.

It's easy to see that  $\lim_{t\to 0}\|w_t\|_{\Hs}=0$. Thus, if we define
$\tilde i(t)=u_0+w_{tt_0}$, then $\lim_{t\to 0}\|\tilde i(t)-u_0\|_{\Hs}=0$. Consequently, $\tilde i\in \Ga$. Therefore, using \eqref{1-8-6}, we obtain
$$\ga\leq \max_{t\in[0,1]}I_{a,f}(\tilde i(t))=\max_{t\in[0,1]}I_{a,f}(u_0+w_{tt_0})<I_{a,f}(u_0)+S^\infty.$$
Hence the claim follows. 

Hence $$I_{a,f}(u_0)<\ga<I_{a,f}(u_0)+S^\infty.$$

Using Ekeland's variational principle, there exists a PS sequence $\{u_n \}$ 
for $I_{a,f}$ at level $\ga$. Doing a standard computation yields $\{u_n\}$ is bounded sequence. Since, by Remark \ref{r:31-7-1}, we have $S^\infty=I_{1,0}(w^*)$, from Proposition \ref{PSP} we can conclude that $u_n\to v_0$, for some $v_0\in \Hs$ such that $(I_{a,f})'(v_0)=0$ and $I_{a,f}(v_0)=\ga$. Further, as
$I_{a,f}(u_0)<\ga$, we conclude $v_0\neq u_0$. 

$(I_{a,f})'(v_0)=0\Longrightarrow v_0$ is a weak solution to \eqref{MAT2}. Combining this with Remark \ref{r:30-7-3}, we conclude the proof of the proposition. 

\end{proof}

\begin{lemma}\lab{l:J1.3}
If $\|f\|_{H^{-s}(\Rn)}<C_pS_1^{\tfrac{p+1}{2(p-1)}}$, then \eqref{J1.3} holds.
\end{lemma}
\begin{proof}
Using the given hypothesis, we can obtain $\eps>0$ such that $\|f\|_{H^{-s}(\Rn)}<C_pS_1^{\tfrac{p+1}{2(p-1)}}-\eps$. Therefore, using Lemma \ref{l:1-8-1}, we have
$$ \prescript{}{H^{-s}}{\langle}f, u{\rangle}_{H^s} \leq \|f\|_{H^{-s}(\Rn)}\|u\|_{\Hs}<\big[C_pS_1^{\tfrac{p+1}{2(p-1)}}-\eps\big]\|u\|_{\Hs}\leq \frac{p-1}{p}\|u\|_{\Hs}^2-\eps\|u\|_{\Hs},$$
for all $u\in U$. Therefore,
$$\frac{p-1}{p}\|u\|_{\Hs}^2-\prescript{}{H^{-s}}{\langle}f,u{\rangle}_{H^s}> \eps\|u\|_{\Hs} \quad\forall\, u\in U.$$
i.e., $$\inf_{U}\bigg[\frac{p-1}{p}\|u\|_{\Hs}^2-\prescript{}{H^{-s}}{\langle}f,u{\rangle}_{H^s}\bigg]\geq \eps\inf_{U}\|u\|_{\Hs}.$$ 
Since, by Remark \ref{r:30-7-1}, we have $\|u\|_{\Hs}$ is bounded away from $0$ on $U$,  the above expression implies
\be\lab{1-8-1}\inf_{U}\bigg[\frac{p-1}{p}\|u\|_{\Hs}^2-\prescript{}{H^{-s}}{\langle}f,u{\rangle}_{H^s}\bigg]>0.\ee
On the other hand, 
\bea\lab{1-8-2'}
\eqref{J1.3}&\iff& C_p\frac{\|u\|_{\Hs}^{\tfrac{2p}{p-1}}}{\|u\|_{L^{p+1}(\Rn)}^{\tfrac{p+1}{p-1}}} - \prescript{}{H^{-s}}{\langle}f,u{\rangle}_{H^s} >0\quad \text{for}\quad \|u\|_{L^{p+1}(\Rn)}=1\no\\
		&\iff& C_p\frac{\|u\|_{\Hs}^{\tfrac{2p}{p-1}}}{\|u\|_{L^{p+1}(\Rn)}^{\tfrac{p+1}{p-1}}} -\prescript{}{H^{-s}}{\langle}f,u{\rangle}_{H^s} >0\quad \text{for}\quad u\in U\no\\
		&\iff& \frac{p-1}{p}\|u\|_{\Hs}^2 -\prescript{}{H^{-s}}{\langle}f,u{\rangle}_{H^s}> 0 \quad\text{for }\quad u\in U.
		\eea
Clearly, \eqref{1-8-1} insures RHS of \eqref{1-8-2'} holds. Hence the lemma follows.

\end{proof}

\vspace{2mm}

{\bf Proof of Theorem \ref{th:ex-f} completed:}
\begin{proof}
Combining Proposition \ref{p:30-7-1} and Proposition \ref{p:31-7-2} with Lemma \ref{l:J1.3}, we conclude the proof of Theorem \ref{th:ex-f}.

\end{proof}

\section{Proof of Theorem \ref{MT}}
In this section we prove multiplicity of positive solutions to \eqref{MAT1} when $a$ satisfies the assumption $\mathbf{(A_1)}$ %and $a\in C(\Rn)$ 
in the spirit of \cite{Adachi}  (also see \cite{Bahri-Li}, \cite{Ad-Ta}).  We aim to obtain the first positive solution of the \eqref{MAT1} through a perturbation of $f \equiv 0$ and exploiting the mountain pass geometry of the functional. For this purpose we need several lemmas and propositions 
along the line proved in \cite{Adachi}. 
\medskip 

We set
\begin{equation}\lab{4-8-1}
\Si:=\{u\in\Hs: \|u\|_{\Hs}=1\} \quad\text{and}\quad \tilde \Si_+:=\{u\in \Si : u_+\not\equiv 0\}.
\end{equation} 
%$$\tilde \Si_+:=\{u\in \Si : u_+\not\equiv 0\}.$$  
Define a modified functional $J_{a, f} : \tilde\Sigma_{+} \longrightarrow \mathbb{R}$ by 
\be\lab{5-8-3}
J_{a,f} := \max_{t>0} I_{a,f}(tv),
\ee
where $I_{a,f}$ is defined as in \eqref{EF-1}. Set, 
\begin{align*}
& \underbar{a} = \inf_{x\in\Rn}a(x)>0, \\
&	\bar{a} = \sup_{x\in\Rn} a(x)=1\no.
\end{align*}

From the definition of $J_{a,f}$, a straight forward computation yields
\be\lab{5-8-1}
J_{a,0}(v)=I_{a,0}\bigg(\bigg(\int_{\Rn}a(x)v_+^{p+1}{\rm d}x\bigg)^{-\frac{1}{p-1}}v\bigg)=\bigg(\frac{1}{2}-\frac{1}{p+1}\bigg)\bigg(\int_{\Rn}a(x)v_+^{p+1}{\rm d}x\bigg)^{-\frac{2}{p-1}}.
\ee
Thus,
$$ \bar{a}^{-\frac{2}{p-1}}J_{1,0}(v) \leq J_{\bar a,0}(v)\leq J_{a,0}(v)\leq J_{\underbar{a},0}(v)=\underbar{a}^{-\frac{2}{p-1}}J_{1,0}(v).$$
Further, as $$\max_{t\in[0,1]} I_{1,0}(tw^*)= I_{1,0}(w^*),$$  where $w^*$ is the unique (radial) ground state solution of \eqref{AT0.8}, we obtain
\be\label{AT1.11}
\bar{a}^{-\tfrac{2}{p-1}}I_{1,0}(w^*)\leq \inf_{v\in\tilde\Sigma_{+}}J_{a,0}(v)\leq \underbar{a}^{-\tfrac{2}{(p-1)}}I_{1,0}(w^*).
\ee

\medskip

\begin{lemma}\label{ATl1.2}
(i) Let $u \in \Hs$ and $\varepsilon \in (0, 1).$ Then there holds 
		\be\label{AT1.12}
		(1-\varepsilon)I_{\tfrac{a}{1-\varepsilon},\;0}(u) - \frac{1}{2\varepsilon}\|f\|_{H^{-s}(\Rn)}^2\leq \Iaf 
		\leq (1+\varepsilon)I_{\tfrac{a}{1+\varepsilon},\;0}(u) + \frac{1}{2\varepsilon}\|f\|_{H^{-s}(\Rn)}^2.
		\ee

(ii) For $v\in\tilde\Sigma_{+}$ and $\varepsilon \in (0, 1)$, there holds 
		\be\label{AT1.13}
		(1-\varepsilon)^{\tfrac{p+1}{p-1}}J_{a,0}(v) - \tfrac{1}{2\varepsilon}\|f\|_{H^{-s}(\Rn)}^2\leq J_{a,f}(v)\leq (1+\varepsilon)^{\tfrac{p+1}{p-1}}J_{a,0}(v) 
		+ \tfrac{1}{2\varepsilon}\|f\|_{H^(\Rn)}^2.
		\ee
(iii) In particular, there exists $d_0>0$ such that if $\|f\|_{H^{-s}(\Rn)}\leq d_0,$
		then, 
		$$\inf_{v\in\tilde\Si_{+}}J_{a,f}(v)>0.$$
\end{lemma}

\begin{proof}
Using Young inequality with $\eps>0$, we can write
 $$|\prescript{}{H^{-s}}{\langle}f,u{\rangle}_{H^s}|\leq \|f\|_{H^{-s}(\Rn)}\|u\|_{\Hs}\leq \frac{\eps}{2}\|u\|_{\Hs}^2+\frac{1}{2\eps}\|f\|_{H^{-s}(\Rn)}^2.$$
Applying the above inequality in the definition of $\Iaf$, we obtain (i). Using (i) in the definition of $J_{a,f}(v)$, we obtain
$$(1-\varepsilon)J_{\frac{a}{1-\eps},0}(v) - \tfrac{1}{2\varepsilon}\|f\|_{H^{-s}(\Rn)}^2\leq J_{a,f}(v)\leq (1+\varepsilon)J_{\frac{a}{1+\eps},0}(v)
		+ \tfrac{1}{2\varepsilon}\|f\|_{H^{-s}(\Rn)}^2.$$
Combining this with \eqref{5-8-1}, we get (ii).		
Finally, substituting \eqref{AT1.11} into (ii) yields (iii).		
\end{proof}

Next, for $v\in \tilde{\Si}_+$, we study properties of the function $\tilde g : [0,\;\infty) \rightarrow \R$ defined as \be\lab{9-8-1}\tilde g(t):= I_{a,f}(tv).\ee
	\begin{lemma}\label{ATl1.3}
(i)  For every $v\in\tilde\Sigma_{+}$, the function $\tilde g$ has at most two critical points in $[0,\;\infty)$.
		
		\medskip 
		
(ii) If $\|f\|_{H^{-s}(\Rn)}\leq d_0$ ($d_0$ is chosen as in Lemma \ref{ATl1.2}), then for any $v\in\tilde\Sigma_{+},$ there exists a unique $\taf>0$ such that 
		$$I_{a,f}\big(\taf v\big) = \Jaf,$$
		where $J_{a,f}$ is defined as in \eqref{5-8-3}.
		Moreover, $\taf>0$ satisfies,
		\be\label{AT1.14}
		\taf > \bigg( p\int_{\Rn}a(x)v_+^{p+1}{\rm d}x\bigg)^{-\tfrac{1}{p-1}}\geq \bigg(pS_{1}^{- \frac{(p+1)}{2}} \bigg)^{-\tfrac{1}{p-1}},
		\ee
		and furthermore
		
		\be\label{AT1.15}
		I''_{a,f}\big(\taf v\big)(v,\;v)<0.
		\ee
		
(iii)  If $\tilde g$ has a critical point different from $\taf$, then it lies in 
$\bigg[0,\;(1-\tfrac{1}{p})^{-1}\|f\|_{H^{-s}(\Rn)}\bigg].$
				
\end{lemma}
This lemma can be proved exactly in the same spirit of \cite[Lemma 1.3]{Adachi}. We skip the details.
%\begin{proof}
%By a straight forward computation, it follows that $\tilde g''>0$ for $t<\bigg( p\int_{\Rn}a(x)v_+^{p+1}dx\bigg)^{-\tfrac{1}{p-1}}$ and $\tilde g''<0$ for $t>\bigg( p\int_{\Rn}a(x)v_+^{p+1}dx\bigg)^{-\tfrac{1}{p-1}}$. Therefore,
%$\tilde g$ has almost two critical points $t_1$ and $t_2$ and they satisfy
%		$$0\leq t_1\leq \bigg( p\int_{\Rn}a(x)v_+^{p+1}dx\bigg)^{-\tfrac{1}{p-1}}\leq t_2<\infty.$$
%\end{proof}
Now we prove the existence of first positive solution in the neighbourhood of $0.$
\subsection{Existence of first solution}
The following proposition provides existence of first positive solution.

\begin{proposition}\label{ATl1.4}
Let $d_0$ be as in Lemma \ref{ATl1.3}. Then there exists $r_1>0$ and $d_1\in (0,\;d_0]$ such that
	
(i) $\Iaf$ is strictly convex in $B(r_1) = \{u\in\Hs \; : \; \|u\|_{\Hs}<r_1\}$.
		
(ii)  If $\|f\|_{H^{-s}(\Rn)}\leq d_1,$ then
		$$
		\inf_{\|u\|_{\Hs}=1}\Iaf >0.
		$$
		Moreover, $I_{a,f}$ has a unique critical point $\ul$ in $B(r_1)$ and it satisfies, 
\be\lab{11-8-1}
		\ul\in B(r_1) \quad \mbox{and} \ I_{a,f}(\ul) =\inf_{u\in B(r_1)} \Iaf.
\ee
		i.e., $\ul$	is a positive solution to $\Iaf$ satisfying \eqref{11-8-1}.\end{proposition}

\begin{proof}
	We begin the proof of part $(i).$ 
	\begin{equation}\label{esti-1}
	I''_{a,f}(u)(h,h) =  \|h\|_{\Hs}^2 -p\int_{\Rn} a(x)u_+^{p-1}h^2 \, {\rm d}x
	\end{equation}
	
	%&\geq& \bigg(1-p\bar{a} \bar{C}_{p+1}^{p+1}\|u\|_{\Hs}^{p-1}\bigg)\|h\|_{\Hs}^2\no
Since $a\leq 1$, using H\"{o}lder inequality and Sobolev inequality, we estimate the second term on the RHS as follows 
	
	\begin{eqnarray*}
	\int_{\Rn} a(x)u_+^{p-1}h^2 \, {\rm d}x \leq \left( \int_{\Rn} |u|^{p+1} \, {\rm d}x\right)^{\frac{p-1}{p+1}} 
	\left(\int_{\Rn} |h|^{p+1} \, {\rm d}x \right)^{\frac{2}{p+1}} &\leq & S_1^{-\frac{p-1}{2}} S_1^{-1}\|u\|^{p-1}_{\Hs} \, \|h\|^{2}_{\Hs} \\
	&=& S_1^{-\frac{p+1}{2}} \, \|u\|^{p-1}_{\Hs} \, \|h\|^{2}_{\Hs}.
	\end{eqnarray*}		
	Thus substituting the above in \eqref{esti-1} we obtain 
	
	$$
	I''_{a,f}(u)(h,\;h) \geq \left( 1 - pS_1^{-\frac{p+1}{2}} \|u\|^{p-1}_{\Hs} \right) \, \|h\|^2_{\Hs}.
	$$	
Therefore, 		
	$I''_{a,\;f}(u)$ is positive definite for $u\in B(r_1)$, with 
	$r_1 =p^{-\frac{1}{p-1}}S_1^\frac{p+1}{2(p-1)}$ and hence $\Iaf$ is strictly convex in $B(r_1).$ This completes the proof of part $(i)$.
	
	\medskip
	
	$(ii)$ Let $\|u\|_{\Hs} = r_1$, then we have
	\bea
	\Iaf = \tfrac{1}{2}\|u\|_{\Hs}^2 - \tfrac{1}{p+1}\int_{\Rn}a(x)u_+^{p+1}\; {\rm d}x-  \prescript{}{H^{-s}}{\langle}f, u{\rangle}_{H^s} &\geq& \tfrac{1}{2}r_1^2-\tfrac{1}{p+1} S_1^{-\frac{p+1}{2}}  r_1^{p+1}-r_1\|f\|_{H^{-s}(\Rn)}\no\\
	&=&\bigg(\frac{1}{2}-\frac{1}{p+1} S_1^{-\frac{p+1}{2}}  r_1^{p-1}\bigg) r_1^2 - r_1\|f\|_{H^{-s}(\Rn)}\no 
	\eea
	
Since, $r_1^{p-1} \, = \, \frac{1}{p} S_1^{\frac{p+1}{2}}$, we obtain	
	$$
	\Iaf \geq \bigg(\frac{1}{2} -\frac{1}{p(p+1)}\bigg)r_1^2-r_1\|f\|_{H^{-s}(\Rn)}.
	$$

	Thus there exists $d_1\in (0,\;d_0]$ such that
	$$
	\inf_{\|u\|_{\Hs} \, = \, r_1}\Iaf >0, \quad  \text{for } \  0<\|f\|_{H^{-s}(\Rn)}\leq d_1.
	$$
Since $\Iaf$ is strictly convex in $B(r_1)$ and $\inf_{\|u\|_{\Hs} \, = \, r_1}\Iaf  >0=I_{a,f}(0)$,  there exists a unique critical point $\ul$ of $I_{a,f}$ in $B(r_1)$ and it satisfies
	\be\lab{2-9-1}
	I_{a,f}(\ul) = \inf_{\|u\|_{\Hs}<r_1}\Iaf<I_{a,f}(0)=0,
	\ee
where the last inequality is due to strict convexity of $I_{a,f}$ in $B(r_1)$.	
	Combining this with Remark \ref{r:30-7-3}, we conclude the proof of the proposition.
\end{proof}

The next proposition characterises all the critical points of $I_{a, f}$ in terms of the functional $J_{a, f}.$ 
\begin{proposition}\label{ATP1.7}
	Let $d_2:=\min\{d_1,\;(1-\tfrac{1}{p})r_1\}>0$, where $d_1$, $r_1$ be as in Proposition \ref{ATl1.4} and suppose that $0<\|f\|_{H^{-s}(\Rn)}\leq d_2$. Then,
	
(i)  $J_{a, f} \in C^1(\tilde\Si_{+},\;\R)$ and
		\begin{equation}\label{AT1.18}
		J'_{a,f}(v)h = \taf I'_{a,f} \big(\taf v\big)h
		\end{equation}
		for all $h\in T_{v}\tilde\Si_{+} = \{h\in\Hs\;|\; \langle h,\;v\rangle_{\Hs} =0\}.$
		
(ii) $ v\in\tilde\Si_{+}$ is a critical point of $\Jaf$ iff $\taf v\in\Hs$ is a critical point of $\Iaf.$
		
(iii)  Moreover, the set of all critical points of $\Iaf$ can be written as 
		\begin{equation}\label{AT1.19}
		\big\{ \taf v\;|\; v\in\tilde\Si_{+},\;J'_{a,f}(v) = 0\big\}\cup \big\{\ul \big\}
\end{equation}		
\end{proposition}

%The proof of the above proposition can be concluded by adopting a similar approach given in \cite[Proposition~1.7]{Adachi}
%with an obvious modification.  For the sake of Brevity we skip the proof. 
\medskip
	\begin{proof}
(i) Let $\tilde g$ be as defined in \eqref{9-8-1}. Then, from Lemma \ref{ATl1.3}, we have  $$\tilde{g}'(t_{a,f}(v))= I'_{a,f}\big(t_{a,f}(v)v\big)(v) =0\quad\text{and}\quad I''_{a,f}\big(\taf v\big)(v,\;v)<0.$$
i.e., $\frac{d^2}{dt^2}\bigg{|}_{t=\taf}I_{a,\;f}(tv)<0.$
Therefore, by implicit function theorem (applying implicit function theorem on the function, $\tilde{F}:(0,\infty)\times \tilde\Si_{+}\to \R,\;\tilde{F}(t,v) = I'_{a,f}(tv)(v)$ which is of class $C^1$),  we can see that $\taf \in C^1(\tilde\Si_{+},\;[0,\;\infty)).$ Consequently, $\Jaf = I_{a,f}(\taf v)\in C^1(\tilde\Si_{+},\;\R)$.\\
Further, as
		\be\label{AT1.20}
		I'_{a,f}(\taf v)(v) = 0,
		\ee
for $h\in T_v\tilde\Si_{+}:=\{h\in\Hs\;|\;\langle h,\;v\rangle_{\Hs} =0\}$, we have
		\bea
		J'_{a,f}(v)h &=& I'_{a,f}(\taf v)\bigg(\taf h+ \langle t'_{a,f}(v),\;h\rangle_{\Hs} v\bigg)\no\\
		&=& \taf I'_{a,f}(\taf v)h+ \langle t'_{a,f}(v),\;h\rangle_{\Hs} I'_{a,f}(t_{a,f}(v)v)(v)\no\\
		&=&\taf I'_{a,f}(\taf v)(h).\no
		\eea
Hence (i) follows.
%		\bea
%		\text{Therefore }J'_{a,f}(v)h = \taf I'_{a,f}(\taf v)h,\;\text{for all } h\in T_v\tilde\Si_{+}.\no
%		\eea

\vspace{2mm}
		
(ii) Applying $(i)$, we have 
$J'_{a,f}(v)=0$ if and only if \be\lab{10-8-1}I'_{a,f}(\taf v)h = 0\quad\forall\; h\in T_v\tilde\Si_{+}.\ee 
Since,  $$\Hs = \mbox{Span} \{v\} \oplus  T_v\tilde\Si_{+},$$
combining \eqref{AT1.20} and \eqref{10-8-1}, (ii) follows.

		\medskip
		
(iii) Suppose that $u\in\Hs$ is a critical point of $I_{a, f}$. Writting $u=tv$ with $v\in \tilde\Si_{+}$ and $t\geq 0$. By lemma \ref{ATl1.3}, we have either $t=\taf$ or $t\leq (1-\tfrac{1}{p})^{-1}\|f\|_{H^{-s}(\Rn)}.$\\
		Thus either $u\in\Hs$ corresponds to a critical point of $J_{a,f}$
		or, $\|u\|_{\Hs} = t\|v\|_{\Hs}=t\leq (1-\tfrac{1}{p})^{-1}d_2\leq r_1$
		By Proposition \ref{ATl1.4}, $\Iaf$ has a unique critical point in $B(r_1)$ and it is $\ul$. Hence the set of all critical points of $\Jaf$ is precisely \eqref{AT1.19}.
	\end{proof}

\medskip

Next we study the Palais-Smale condition for $\Jaf$. 

\begin{proposition}\label{ATP1.9}
	Suppose $0<\|f\|_{H^{-s}(\Rn)}\leq d_2,$ where $d_2>0$ is as found in Proposition \ref{ATP1.7}. Then,
	
(i) $J_{a,f}(v_j)\to\infty$ whenever $dist_{\Hs}(v_j,\;\partial\tilde\Si_{+}) \overset{j}\to 0$, where 
$$dist_{\Hs}(v_j,\;\partial\tilde\Si_{+}):= \inf\{\|v_j-u\|_{\Hs}\;:\;u\in\Sigma,\;u_+\equiv 0\}.$$

(ii) Suppose that $\{v_j\}_{j=1}^{\infty}\subset \tilde\Si_{+}$ satisfies as $j\to\infty$
\be \lab{10-8-3}
	J_{a,f}(v_j)\to c, \quad\text{for some }c>0,\ee
\be\lab{10-8-4}\|J'_{a,f}(v_j)\|_{T_v^*\tilde\Si_{+}} \equiv\sup\{J'_{a,f}(v_j)h : \;h\in T_{v_j}\tilde\Si_{+},\;\|h\|_{\Hs}=1\}\To 0.\ee
	Then there exists a subsequence- still we denote by $\{v_j\}$, a critical point $u_0(x)\in\Hs$ of $\Iaf$, an integer $l\in\mathbb{N}\cup\{0\}$ and $l$ sequences of points $\{y_j^{(1)}\},\ldots,\;\{y_j^{(l)} \}\subset \Rn$, critical points $w_k\in\Hs$ \textup{($k=1,2,\cdots l$)} of \eqref{AT0.8} such that 
	
	\medskip
	
	\begin{enumerate}
	
	\item $\;|y_j^k|\to\infty$ as $j\to\infty$, for all $k=1,\;2,\ldots,\;l$.
	
	\medskip
	
	\item $\;|y_j^{(k)}-y_j^{(k')}|\to\infty$ as $j\to\infty$ for $k\neq k'$.
	
	\medskip
	
	\item $\; \bigg{\|}v_j(x)- \frac{u_0(x)+\sum_{k=1}^{l}w_k(x-y_j^k)}{\|u_0(x)+\sum_{k=1}^{l}w_k(x-y_j^k)\|_{\Hs}}\bigg{\|}_{\Hs}\to 0$ as $j\to\infty.$
	
 \medskip	
	
\item $ J_{a,f}(v_j)\to I_{a,f}(u_0) + \sum_{k=1}^{l}I_{1,0}(w_k)$ as $j\to\infty$.
\end{enumerate}
	
\end{proposition}
\begin{proof}
(i) Using \eqref{AT1.13} and \eqref{5-8-1},  for any $\varepsilon \in (0,\;1)$, we have
	\bea
	J_{a,f}(v_j)&\geq& (1-\varepsilon)^{\tfrac{p+1}{p-1}}J_{a,0}(v_j)-\frac{1}{2\varepsilon}\fhs^2 \no\\
	&\geq& (1-\varepsilon)^{\tfrac{p+1}{p-1}}(\tfrac{1}{2}-\tfrac{1}{p+1})\bigg(\int_{\Rn}a(x)v_{j{+}}^{p+1}\;{\rm d}x\bigg)^{-\tfrac{2}{p-1}}-\frac{1}{2\varepsilon}\fhs^2\no
	\eea
	Since, $dist(v_j,\;\partial\tilde\Si_{+})\to 0$ implies
	$(v_j)_{+}\to 0 \;\text{in }\Hs$. Therefore,  
	$(v_j)_{+}\to 0 \;\text{in }L^{p+1}(\Rn)$.
Consequently,
$$\bigg{|}\int_{\Rn}a(x)(v_j)_{+}^{p+1}\;{\rm d}x\bigg{|}\leq \|a\|_{L^{\infty}(\Rn)}\int_{\Rn}(v_j)_+^{p+1}\;{\rm d}x\to 0 \;\text{as}\;j\to\infty.$$
Therefore, $$J_{a,f}(v_j)\To\infty \quad\text{as }\quad dist_{\Hs}(v_j,\;\partial \tilde\Si_{+})\To 0.$$ Hence (i) follows.

(ii) From \eqref{AT1.14} and \eqref{AT1.18} we have,
	\bea
	\big{\|}I'_{a,f}\big(t_{a,f}(v_j)v_j\big)\big{\|}_{\hms} &=& \frac{1}{t_{a,f}(v_j)}\|J'_{a,f}(v_j)\|_{T^*_{v_j}\tilde\Si_{+}}\no\\
	&\leq& \bigg(pS_{1}^{-\tfrac{p+1}{2}}\bigg)^{\tfrac{1}{p-1}}\|J'_{a,f}(v_j)\|_{T^*_{v_j}\tilde\Si_{+}}\overset{j}\To 0\no 		
	\eea
	We also have, $I_{a,f}(t_{a,f}(v_j)v_j) = J_{a,f}(v_j)\to c$ as $j\to\infty$.
	Applying Palais-Smale result for $\Iaf$ (Proposition \ref{PSP} %and Remark \ref{RMKPSP}
	), we conclude (ii).
\end{proof}
As a consequence to the above Proposition \ref{ATP1.9}, we have,
\begin{corollary}\label{corr-ps}
	Suppose that $0<\fhs\leq d_2$, where $d_2>0$ is as found in Proposition \ref{ATP1.7}. Then $J_{a,f}$ satisfies $(PS)_c$ at level $$c < I_{a,f}\big(\ul\big)+I_{1,0}(w^*),$$
	where $w^*$ is the unique ground state solution of \eqref{AT0.8}. 
\end{corollary}
Here we say that $\Jaf$ satisfies $(PS)_c$ if and only if for any sequence $\{v_j\}\subseteq \tilde\Si_{+}$ satisfying \eqref{10-8-3} and \eqref{10-8-4}
has a strongly convergent subsequence in $\Hs$.

\begin{proof}
By \eqref{2-9-1},
\be\lab{14-8-1}I_{a,f}(\ul) <0.\ee
On the other hand, from \eqref{AT1.19} we see that apart from $\ul$, all critical points of $I_{a,f}$ corresponds to a critical point $J_{a,f}$. So, if $u_1$ is a critical point of $I_{a,f}$, there exists $v_1\in\tilde{\Si}_{+}$ such that $I_{a,f}(u_1) = J_{a,f}(v_1) >0$ (here we have used $(iii)$ of Lemma \ref{ATl1.2}). Hence, 
	\be\lab{12-8-1}I_{a,f}\big(\ul\big) = \inf\bigg\{I_{a,f}(u_0)\;\bigg{|}\;u_0\in\Hs\;\text{is a critical point of  }I_{a,f}\bigg\}.\ee
Consequently, $I_{a,f}\big(\ul\big)+I_{1,0}(w^*)\leq I_{a,f}(u_0)+\sum_{i=1}^{l}I_{1,0}(w_i)$, for any critical point $u_0$ of $I_{a,f}$ and $l\geq 1$, where $w^*$ is the unique positive ground state solution of \eqref{AT0.8} and $w_i$ are positive solutions of \eqref{AT0.8} . From Proposition \ref{ATP1.9}, we know that if PS sequence for $J_{a,f}$ breaks down at level $c$, then $c$ must be of the form $I_{a,f}(u_0)+\sum_{i=1}^{l}I_{1,0}(w_i)$, where $u_0$ is any critical point of $I_{a,f}$ and $l\in \N\cup\{0\}$. Thus, if $l=0$ and $u_0=\ul$, then applying \eqref{14-8-1} to the Proposition \ref{ATP1.9}(ii)(4), we have $\lim_{j\to\infty} J_{a,f}(v_j)=I_{a,f}(\ul)< 0$. On the other hand, from Lemma \ref{ATl1.2}(iii) we have $\lim_{j\to\infty} J_{a,f}(v_j)>0$, which gives a contradiction. Therefore,  $l=0$ and $u_0=\ul$ can not happen together. Now, if $l=0$ and $u_0\neq \ul$, then from Proposition~\ref{ATP1.9}(ii)(3), it follows $v_j \rightarrow \frac{u_0}{||u_0||}$ in $\Hs$. Hence 
the Palais-Smale condition at level $c$ is satisfied.   
Thus the lowest level of breaking down of $(PS)_c$ is $I_{a,f}(\ul)+I_{1,0}(w^*)$. Hence the corollary follows.

\end{proof}

%%%%%%%%%%%%%%%%%%%%%%%%%%%%%%%%%%%%%%%%%%%%%%%%%%%%%%%%%%%%%%%%%%%%%%%%%%%%%%%%%%%%%%%%%%
\medskip 

\subsection{Existence of second and third solution} In this subsection, our main aim is to show the existence of second and third positive solution. To this aim we shall use Lusternik-Schnirelman Category theory and a careful analysis of Palais-Smale sequence to prove multiplicity result. We use the following notation. 

\be\lab{12-8-5}
\[J_{a,f}\leq c\] =\{v\in\tilde\Sigma_{+}\;|\;\Jaf\leq c\},
\ee
for $c\in\R$. As explained before in order to find the critical points of $\Jaf$, we show for a sufficiently small $\varepsilon>0$,
\be\label{AT2.1}
{\rm cat} \bigg(\[J_{a,f}\leq I_{a,f}(\ul)+I_{1,\;0}(w^*)-\varepsilon\] \bigg)\geq 2,
\ee
where cat denotes Lusternik-Schnirelman Category.

\medskip 

Now we prove a very delicate energy estimate which plays a pivotal role in the proof 
of existence of critical points. 

\begin{proposition}\label{main-energy-estimate}
	Let $a$ be as in Theorem \ref{MT} and $f$ be a nonnegative nontrivial functional in $H^{-s}(\Rn)$ with $\fhs\leq d_2$, where $d_2>0$ is as found in Proposition \ref{ATP1.7}. Then there exists $R_0 >0$ such that
	
	\be\label{main-AT2.2}
	I_{a,f}(\ul+tw^*(x-y)) <I_{a,f}(\ul)+I_{1,0}(w^*),
	\ee
	for all $|y|\geq R_0$ and $t>0$. Here $w^*$ is the unique ground state solution of \eqref{AT0.8}.	
\end{proposition}

\begin{proof}
	It is easy to see that $I_{a,f}\big(\ul+tw^*(x-y)\big) \To I_{a, f}(\ul)<0,$ as $t \rightarrow 0,$ which follows from 
	the continuity of $I_{a,f}$. It also follows
$$I_{a,f}\big(\ul+tw^*(x-y)\big) \rightarrow -\infty,\quad\text{as} \quad {t\to\infty}.
	$$
	From these two facts, there exist $m,\, M$ with $0 < m < M$ such that 
$$I_{a,f}(\ul+tw^*(x-y)) <I_{a,f}(\ul)+I_{1,\;0}(w^*) \quad \mbox{for all} \ t \in (0, m) \cup (M, \infty).
	$$
In view of above it is enough to prove \eqref{main-AT2.2} for all $t \in  [m, M].$ We can write 
\begin{align}\label{main-esti-energy-1}
	& I_{a,f}\big(\ul+tw^*(x-y)\big)  \no\\
	&= \frac{1}{2}\big{\|}\ul +tw^*(x-y)\big{\|}_{\Hs}^2  \no\\
	&\qquad -\frac{1}{p+1}\int_{\Rn} a(x)\big(\ul +tw^*(x-y)\big)^{p+1} \, {\rm d}x \no\\
	&\qquad\qquad -\prescript{}{H^{-s}}{\big\langle} f,\, \big(\ul +tw^*(x-y)\big){\big\rangle}_{H^s} \no \\
	& = \frac{1}{2}\|\ul\|_{\Hs}^2+\frac{t^2}{2}\|w^*\|_{\Hs}^2+t\langle \ul,\;w^*(x-y)\rangle_{\Hs} \no\\
	&\qquad -\frac{1}{p+1} \int_{\Rn} a(x) (\ul)^{p+1} \, {\rm d}x  - \frac{t^{p+1}}{p+1}\int_{\Rn} a(x)w^*(x-y)^{p+1} \, {\rm d}x \no \\
	&\qquad\qquad -\frac{1}{p+1}\int_{\Rn} a(x) \bigg\{ (\ul + t w^*(x-y))^{p+1} - (\ul)^{p+1} - t^{p+1}w^*(x-y)^{p+1} \bigg\} {\rm d}x \no\\
	& \qquad\qquad\quad -\prescript{}{H^{-s}}{\big\langle} f,\, \big(\ul +tw^*(x-y)\big){\big\rangle}_{H^s}.
	\end{align}
	Also we have for all $ h \in\Hs,$
	
	$$
	0 = I'_{a,f}\big({\ul}\big)(h) = \langle \ul,\;h\rangle_{\Hs}-\int_{\Rn} a(x)(\ul)^p h \, {\rm d}x  -\prescript{}{H^{-s}}{\langle} f,\,h {\rangle}_{H^s} , 
	$$
	which in turn implies 
	
	$$
	\langle \ul,\;h\rangle_{\Hs} = \int_{\Rn}a(x) (\ul)^p h \, {\rm d}x
	+ \prescript{}{H^{-s}}{\langle} f,\,h {\rangle}_{H^s}.
	$$
	
	Now by setting $h = t w^*(x - y)$ in above expression, we obtain 
	
	$$
	t \langle \ul,\;w^*(x-y)\rangle_{\Hs} = t \int_{\Rn}a(x) (\ul)^p w^*(x-y) \, {\rm d}x
	+t\prescript{}{H^{-s}}{\langle} f,\,w^*(x-y) {\rangle}_{H^s}.
	$$
	
	Thus using above and the rearranging the terms in \eqref{main-esti-energy-1} we have 
	\bea\lab{11-8-3}
	& I_{a,f}\big(\ul+tw^*(x-y)\big) = I_{a,f}\big(\ul\big)+I_{1,0}(tw^*)+\frac{t^{p+1}}{p+1}\displaystyle\int_{\Rn}\big(1-a(x)\big)w^*(x-y)^{p+1} \, {\rm d}x\no\\
	& -\frac{1}{p+1}\displaystyle\int_{\Rn} a(x)\bigg\{\big(\ul + tw^*(x-y) \big)^{p+1} - (\ul)^{p+1}\no\\
	& - t(p+1)(\ul)^p w^*(x-y) - t^{p+1} w^*(x-y)^{p+1} \bigg\} \, {\rm d}x \no\\
	& = I_{a,f} \big(\ul \big) + I_{1,0}(tw^*) + (I) - (II),
	\eea
	where $$(I) := \frac{t^{p+1}}{p+1}\int_{\Rn}\big(1-a(x)\big)w^*(x-y)^{p+1}\, {\rm d}x$$
	and
	
	 \begin{align*}
	(II) & := \frac{1}{p+1}\int_{\Rn} a(x)\bigg\{\big(\ul \, + \,tw^*(x-y) \big)^{p+1} - (\ul)^{p+1}  \\
	& - t(p+1)(\ul)^p w^*(x-y) - t^{p+1} w^*(x-y)^{p+1} \bigg\} \, {\rm d}x.
	\end{align*}
	Therefore the proof will be completed if we can show $I < II.$ To this aim let us recall a standard fact from calculus. The following inequalities
	hold true
	
	\medskip 
	
	\begin{itemize}
		\item  $(s+t)^{p+1}-s^{p+1}-t^{p+1}-(p+1)s^pt\geq 0$ for all $(s,\;t)\in[0,\;\infty)\times [0,\;\infty).$
		
		\medskip 
		
		\item  For any $r>,0$ we can find a constant $A(r) > 0$ such that
		$$(s+t)^{p+1}-s^{p+1}-t^{p+1}-(p+1)s^pt\geq A(r)t^2
		$$
		for all $(s,\;t)\in[r,\;\infty)\times [0,\;\infty)$.
	\end{itemize}
	
	Using the above inequality $II$ can be estimated as follows :  setting $r:=\min_{|x|\leq 1}\ul>0$, $A:=A(r)$, we have
	
	\begin{align}\label{esti-ii}
	(II) & \geq \frac{1}{p+1} \int_{|x| \leq 1} a(x) At^2 (w^*)^2(x-y) \, {\rm d}x 
	\geq \frac{m^2 \, \underbar{a}A}{p+1} \int_{|x|\leq 1} (w^*)^2(x - y) \, {\rm d}x  \notag \\
	& \geq \frac{C \, m^2 \, \underbar{a}A}{p+1} \int_{|x|\leq 1} \frac{{\rm d}x}{(1 + |x - y|^{N+2s})^2}
	\geq  \frac{C \, m^2 \, \underbar{a}A}{p+1} |y|^{-2(N + 2s)}, 	\end{align}
	where in the last inequality we have used the fact that for $|x| \leq 1,$ there exists $R > 0$ with
	$|y| > R,$  such that 
\be\lab{11-8-2}1 + |x - y|^{N+ 2s} \approx |y|^{(N + 2s)}.\ee
On the other hand, 	
	\begin{align}\label{esti-iii}
	(I)= \frac{t^{p+1}}{p+1}\int_{\Rn}\big(1-a(x)\big)w^*(x-y)^{p+1}\, {\rm d}x 
	& \leq \frac{t^{p+1}}{p +1} \int_{\Rn} 
	\frac{C}{1+|x|^{\mu(N+2s)}}\bigg\{\frac{C_2}{1+|x-y|^{N+2s}}\bigg\}^{p+1} \, {\rm d}x \notag \\
	& \leq \frac{C M^{p+1}}{p+1} \int_{\Rn} \frac{{\rm d}x}{(1 + |x|^{\mu(N + 2s)})( 1 + |x - y|^{N+ 2s})^{p+1}}.
	\end{align}
	
{\bf Claim:} $\displaystyle\int_{\Rn} \frac{{\rm d}x}{(1 + |x|^{\mu(N + 2s)})( 1 + |x - y|^{N+ 2s})^{p+1}}\leq \frac{C'}{|y|^{(N+2s)(p+1)}}$ for $|y|$ large enough.

\vspace{2mm}

Using the claim, first we complete the proof. Clearly combining the above claim with \eqref{esti-ii}, it immediately follows that  there exists $R_0>R> 0$ large enough such that $$(I) <(II)\;\text{for } |y|\geq R_0,$$
as $p+1>2$. Hence \eqref{main-AT2.2} follows from \eqref{11-8-3}.

\vspace{2mm}

 Therefore, we are left to prove the claim. 
 
 \vspace{2mm}
 
 Note that, to prove the claim, it is enough to show that 
	$$\int_{\Rn} \frac{|y|^{(p+1)(N + 2s)}}{(1 + |x|^{\mu(N + 2s)})( 1 + |x - y|^{N+ 2s})^{p+1}} \, {\rm d}x \leq C(N, M, \underbar{a}, p).
$$ 
Therefore we estimate LHS of the above inequality:
\begin{eqnarray}
	\int_{\Rn} \frac{|y|^{(p+1)(N + 2s)}}{(1 + |x|^{\mu(N + 2s)})( 1 + |x - y|^{N+ 2s})^{p+1}} \, {\rm d}x &\leq &
	  \underbrace{\int_{\Rn} \frac{|x-y|^{(p+1)(N + 2s)}}{(1 + |x|^{\mu(N + 2s)})( 1 + |x - y|^{N+ 2s})^{p+1}} \, {\rm d}x}_{:= J_1}\no\\
	 &&\quad+\underbrace{\int_{\Rn} \frac{|x|^{(p+1)(N + 2s)}}{(1 + |x|^{\mu(N + 2s)})( 1 + |x - y|^{N+ 2s})^{p+1}} \, {\rm d}x}_{:= J_2}.\no
	\end{eqnarray}
Clearly,
$$J_1\leq \int_{\Rn}\frac{{\rm d}x}{1 + |x|^{\mu(N + 2s)}}=C(N,\mu),$$	
since $\mu>\frac{N}{N+2s}$.  On the other hand, using \eqref{11-8-2}, we estimate	
$$J_2\leq\int_{\Rn} \frac{|x|^{(p+1)(N + 2s)}}{1 + |x|^{\mu(N + 2s)}} \, {\rm d}x =C(N,\mu, p),$$
since $\mu>(p+1)+\frac{N}{N+2s}$ (by hypothesis of the proposition). Combining the above estimates, claim follows. Hence we conclude the proof of the Proposition. 

\end{proof}

We further need several preparatory lemmas and propositions along with the key energy estimate \eqref{main-AT2.2}
to prove existence of second and third positive solution. The results below are along the line of 
\cite{Adachi, Bahri-Li}.

\medskip 

We begin with the properties of the functional $J_{a, 0}$ under the condition ${\bf{(A_1)}}.$

\begin{lemma}\label{ATP1.11}
Let $a$ be as in Theorem \ref{MT} and $w^*$ is unique ground state solution of \eqref{AT0.8}. Then there holds 

(i) $\inf_{v \in \tilde\Si_{+}} \, J_{a, 0}(v) = I_{1, 0}(w^*)$.
		
(ii) $\inf_{v \in \tilde\Si_{+}} \, J_{a, 0}(v)$ is not attained. 
		
(iii) $J_{a, 0}(v)$ satisfies $(PS)_{c}$ for $c \in (-\infty, I_{1, 0}(w^*))$. %\cup (I_{1, 0}(w^*), 2 I_{1, 0}(w^*)).$
\end{lemma}

\begin{proof}
	
(i) Using \eqref{AT1.11}, we have
	
	$$\inf_{v \in \Sigma_{+}} \, J_{a, 0}(v) \geq I_{1, 0}(w^*).$$
	
Define $w_l(x) = w^*(x + l e)$, where $e$ is an unit vector in $\Rn$. Using Lemma~\ref{ATl1.3}, corresponding 
	to $\bar{w}_{l} =\frac{w_l}{\|w_l\|_{\Hs}} \in \tilde\Sigma_{+}$ there exists an unique $t_{a, 0}(\bar{w}_{l} )$ such that 
	
	$$
	J_{a, 0}\bigg(\frac{w_l}{\|w_l\|_{\Hs}}\bigg) = I_{a, 0} \bigg(t_{a, 0}(\bar{w}_{l} )\frac{w_l}{\|w_l\|_{\Hs}} \bigg).
	$$
	Now let us compute 
	
	\begin{align*}
	I_{a, 0} \bigg(t_{a, 0}(\bar{w}_{l} )\frac{w_l}{\|w_l\|_{\Hs}} \bigg)
	& = \frac{t^2_{a, 0}(\bar{w}_{l} )}{2} \|\bar{w}_l\|_{\Hs}^2 - \frac{t^{p+1}_{a, 0}(\bar{w}_{l} )}{p+1} \int_{\Rn} a(x) (\bar{w}_l)^{p+1} \, {\rm d}x. \notag \\
	% & = \frac{t^2_{a, 0}(\bar{w}_{l} )}{2} - \frac{t^{p+1}_{a, 0}(\bar{w}_{l} )}{(p+1)\|w_l\|_{\Hs}^{(p+1)} \int_{\Rn} a(x)(w_l)^{p+1} \, {\rm d}x \notag \\
	\end{align*}
	Moreover from direct computation, we find an explicit form of  $t_{a, 0}(\bar{w}_l)$ which is given by 
	
	$$
	t_{a,0}(\bar{w}_l) = \bigg(\int_{\Rn} a(x)\bar{w}_l^{p+1}{\rm d}x\bigg)^{-\tfrac{1}{p-1}} \xrightarrow{l\to\infty} \bigg(\frac{\|w^*\|_{\Hs}}{\|w^*\|_{L^{p+1}(\Rn)}}\bigg)^{\tfrac{p+1}{p-1}},
	$$
	the last limit follows since $a(x)\to 1$ as $|x|\to\infty$.
Hence,
	\begin{align}
	J_{a,0}(\bar{w}_l) \xrightarrow{l\to\infty}&\frac{1}{2}\bigg\{\frac{\|w^*\|_{\Hs}}{\|w^*\|_{L^{p+1}(\Rn)}}\bigg\}^{\tfrac{2(p+1)}{(p-1)}}-\frac{1}{p+1}\bigg(\bigg\{\frac{\|w^*\|_{\Hs}}{\|w^*\|_{L^{p+1}(\Rn)}}\bigg\}^{\tfrac{(p+1)^2}{(p-1)}}\times\frac{\|w^*\|^{p+1}_{L^{p+1}(\Rn)}}{\|w^*\|_{\Hs}^{p+1}}\bigg)\notag\\
	&= \big(\frac{1}{2}-\frac{1}{p+1}\big)\|w^*\|_{L^{p+1}(\Rn)}^{p+1}
	= I_{1,0}(w^*).\no
	\end{align} 
	Hence $(i)$ follows.
	
	\medskip

(ii) Let us assume on the contrary that there exists $v_{0} \in \tilde\Sigma_+$
such that $J_{a, 0}(v_0) = \inf_{v \in \tilde\Sigma_+} J_{a, 0}(v) = I_{1, 0}(w^*).$
Define, the Nehari manifold $\mathcal{N}$ as
$$\mathcal{N} := \left\{ u \in \Hs : \langle(I_{1, 0})'(u), u\rangle =0 \right\}.$$
From a straight forward computation, it is easy to see that there exists
$t_{v_0}  >0$ such that $t_{v_0} v_0 \in \mathcal{N}.$ 
Further, observe that for any $v \in \mathcal{N}$, it holds
$$
I_{1, 0}(v) = \frac{p-1}{2(p+1)} ||v||^2_{\Hs}\geq \frac{p-1}{2(p+1)}S_1^\frac{p+1}{p-1},
$$
where $S_1$ is as defined in \eqref{17-7-5}. therefore, it follows from  Remark~\ref{r:31-7-1} that $I_{1, 0}(v) \geq I_{1, 0}(w^*)$ for all $v \in \mathcal{N}.$ Moreover 
$w \in \mathcal{N}$ and hence 
$$\inf_{v \in \mathcal{N}} I_{1, 0}(v) = I_{1, 0}(w^*). $$	Therefore,
 \begin{align}\lab{14-8-2}
	I_{1, 0}(w^*) = J_{a, 0} (v_0) & := {\max}_{t> 0} I_{a, 0}(t v_0) \geq I_{a, 0}(t_{v_0} v_0) \notag \\
	& = \frac{t^2_{v_0}}{2} \| v_0 \|_{\Hs}^2 - \frac{t^{p+1}_{v_0}}{p+1} \int_{\Rn} a(x) (v_0)_{+}^{p+1} \, {\rm d}x \notag \\
	& = \frac{t^2_{v_0}}{2} \| v_0 \|_{\Hs}^2 - \frac{t^{p+1}_{ v_0}}{p+1} \int_{\Rn}  (v_0)_{+}^{p+1} \, {\rm d}x  +  \frac{t^{p+1}_{v_0}}{p+1} \int_{\Rn}(1- a(x)) (v_0)_{+}^{p+1} \, {\rm d}x \notag\\
	& = I_{1, 0}(t_{v_0} v_0) \, + \,  \frac{t^{p+1}_{v_0}}{p+1} \int_{\Rn}(1- a(x)) (v_0)_{+}^{p+1} \, {\rm d}x\notag\\
		& \geq I_{1, 0}(w^*) \, + \,  \frac{t^{p+1}_{v_0}}{p+1} \int_{\Rn}(1- a(x)) (v_0)_{+}^{p+1} \, {\rm d}x.
	\end{align}

	The above inequality and ${\bf (A1)}$ implies 
	\be\lab{14-8-4} \frac{t^{p+1}_{ v_0}}{p+1}\displaystyle \int_{\Rn}(1- a(x)) (v_0)_{+}^{p+1} \, {\rm d}x = 0.\ee Therefore 
	\begin{equation}\label{vequal0}
	(v_0)_+ \equiv 0 \quad \mbox{in} \  \{ x \in \Rn : a(x) \neq 1 \}.
	\end{equation}
Moreover, substituting \eqref{14-8-4} into \eqref{14-8-2}, we see that inequality in \eqref{14-8-2} becomes an equality there. Therefore, 
$$\inf_{\mathcal{N}} I_{1, 0}(v)=I_{1, 0}(w^*) = I_{1, 0}(t_{v_0}v_0).$$
Thus $t_{v_0}v_0$ is a constraint critical point of $I_{1,0}$. Therefore using Lagrange multiplier and maximum principle (as before) we conclude that $t_{v_0}v_0 > 0$ which in turn implies $v_0 > 0$ in $\Rn.$ This contradicts \eqref{vequal0}. Hence (ii) holds.

\medskip 
	
(iii) From Proposition \ref{ATl1.4}, we know that $\ul$ is the unique critical point of $I_{a,f}$ in $B(r_1)$. Therefore,  $u_{locmin}(a, 0;x) = 0.$ Consequently,  it follows from 
Corollary~\ref{corr-ps} that Palais-Smale condition for $J_{a,0}$ is 
satisfied at the level $c < I_{1, 0}(w^*).$ 
%On the other hand, we know that from Proposition \ref{ATP1.9} that the 
%$(PS)_c$ condition for $J_{a, 0}$ breaks down only for 
%$$
%c = I_{a, 0}(u_0) + lI_{1, 0}(w),
%$$
%where $u_0$ is a critical point of $I_{a,0}$ and $l\in\N\cup\{0\}$. Therefore, 
%after $I_{1,0}(w)$, next level $c$, where $(PS)_c$ for $J_{a,0}$ may break down is either $2I_{1, 0}(w)$ or $I_{a, 0}(u_0) + I_{1, 0}(w)$, where $u_0\neq \ul$. We observe that if  $u_0 \neq u_{locmin}(a, 0;x),$ then $u_0$ is a non trivial critical point for $I_{a, 0}.$ Therefore using Proposition~\ref{ATP1.7}, 
%there exists $t_{a, 0} >0$ such that $u_0 = t_{a, 0}v_0$ and $v_0$ is a critical point of $J_{a, 0}$. Moreover, by Lemma \ref{ATl1.3}(ii) 
%$$
%J_{a, 0}(v_0) = I_{a, 0}(u_0).
%$$ 
%We know from part (i) and (ii) that $J_{a, 0}(v_0) > I_{1, 0}(w),$ which in turn implies $I_{a, 0}(u_0) > I_{1, 0}(w).$ Therefore 
%it immediately yields that $(PS)_c$  is 
%satisfied for  $ I_{1, 0}(w) < c < 2I_{1, 0}(w).$ 

This completes the proof.

\end{proof}

The following property of $J_{a,0}(v)$ is important to obtain multiplicity of solutions of \eqref{MAT1}
\begin{lemma}{\textbf{(Center of mass)}}\label{ATL1.12}\\
	Let $a$ be as in Theorem \ref{MT}. Then there exists a constant $\delta_{0}>0$ such that if $J_{a,0}(v)\leq I_{1,0}(w^*)+\delta_{0}$, then 
	\be\label{AT1.23}
	\int_{\Rn}\frac{x}{|x|}|v(x)|^{p+1}\;{\rm d}x \neq 0.
	\ee
\end{lemma}
%\begin{proof}
%	Since $\inf_{v\in\Sigma_{+}}J_{a,\;0}(v) = I_{1,\;0}(w)$  is not attained, it follows from $(ii)$ of Proposition \eqref{ATP1.9} that for any $R\geq 1,$ there exists an $\varepsilon =\varepsilon(R)>0$ with the following property:\\
%	if $v\in\Sigma_{+}$ satisfies
%	\bea
%	J_{a,\;0}(v) &\leq& I_{1,\;0}(w)+\varepsilon\no\\
%	\|J'_{a,\;0}(v)\|_{T^*_v\Sigma_{+}} &\leq& \varepsilon\no
%	\eea 
%	then
%	\be\label{AT1.24}
%	\bigg{\|}v-\frac{w(x-y)}{\|w\|_{\Hs}}\bigg{\|}_{\Hs}\leq \frac{1}{R}\;\text{for some }|y|\geq R
%	\ee
%	We choose $R\geq 1$ sufficiently large so that \eqref{AT1.24} implies \eqref{AT1.23}.\\
%Suppose $v\in\Sigma_{+}$ satisfies $J_{a,\;0}(v)\leq I_{1,\;0}(w)+\delta_{0}$. Then by Ekeland's variational principle (cf. \cite{Willem} Michel WIllem, Minimax Theorems) there exists $\tilde{v}\in\Sigma_{+}$ such that
%\bea
%\|v- \tilde{v}\|_{\Hs} &\leq& \sqrt{\delta_{0}}\no\\
%\|J'_{a,\;0}(\tilde{v})\|_{T^*_{\tilde{v}}\Sigma_{+}} &\leq& \sqrt{\delta_{0}}\no\\
%J_{a,\;0}(\tilde{v}) &\leq& I_{1,\;0}(w)+\delta_{0}\no	    
%\eea
%Choosing $\delta_{0}\leq\min\{\varepsilon(2R)^2,\;\varepsilon(2R),\;\frac{1}{4R^2}\}$, we have,
%\bea
%\bigg{\|}v-\frac{w(\bullet-y)}{\|w\|_{\Hs}}\bigg{\|}_{\Hs}&\leq& \|v-\tilde{v}\|_{\Hs}+\bigg{\|}\tilde{v}-\frac{w(\bullet-y)}{\|w\|_{\Hs}}\bigg{\|}_{\Hs}\no\\
%&\leq& \sqrt{\delta_{0}}+\frac{1}{2R}\leq \frac{1}{R}\no
%\eea
%for some $|y|\geq 2R$. Thus we have \eqref{AT1.23}.
%\end{proof}
\begin{proof}
	Suppose the conclusion is not true. Then there exists a sequence $\{v_n\}\subset \tilde\Sigma_{+}$ such that
	$$J_{a,0}(v_n) \leq I_{1,0}(w^*)+\tfrac{1}{n}\;\text{and }\int_{\Rn}\frac{x}{|x|}|v_n|^{p+1}{\rm d}x \xrightarrow{n\to\infty} 0.$$
	Since, by Lemma \ref{ATP1.11}, we have $\inf_{v \in \tilde\Si_{+}} \, J_{a, 0}(v) = I_{1, 0}(w^*)$ and the infimum is not attained, applying Ekeland's variational principle, there exists $\tilde{v}_n\subset \tilde\Sigma_{+}$ such that
	\begin{align}
	\|v_n - \tilde{v}_n\|_{\Hs}&\xrightarrow{n\to\infty} 0\no\\
	J_{a,0}(\tilde{v}_n)&\leq J_{a,0}(v_n) = I_{1,0}(w^*)+\tfrac{1}{n}\no\\
	J'_{a,0}(\tilde{v}_n)&\xrightarrow{n\to\infty} 0 \;\text{in }H^{-s}(\Rn).\no
	\end{align}
	Therefore, $\{\tilde{v}_n\}$ is a Palais Smale sequence for $J_{a,0}$ at the level $I_{1,0}(w^*)$. Applying Proposition \ref{ATP1.9}, we get $\{y_n\}\subset \Rn$ such that $|y_n|\xrightarrow{n}\infty$ and
	$$\bigg{\|}\tilde{v}_n-\frac{w^*(x-y_n)}{\|w^*(x-y_n)\|_{\Hs}}\bigg{\|}_{\Hs}\xrightarrow{n\to\infty} 0.$$
	Therefore, 
	\begin{align}
	\bigg{\|}v_n-\frac{w^*(x-y_n)}{\|w^*(x-y_n)\|_{\Hs}}\bigg{\|}_{\Hs}&\leq \|v_n - \tilde{v}_n\|_{\Hs}\no\\
	&+\bigg{\|}\tilde{v}_n-\frac{w^*(x-y_n)}{\|w^*(x-y_n)\|_{\Hs}}\bigg{\|}_{\Hs}
	\xrightarrow{n\to\infty}0\no.
	\end{align}
	Therefore the above yields 
	
	\begin{align}
	\circ({\bf{1}}) &= \int_{\Rn}\frac{x}{|x|}|v_n|^{p+1}{\rm d}x = 
	\int_{\Rn}\frac{x}{|x|}\bigg(\frac{w^*(x-y_n)}{\|w^*(x-y_n)\|_{\Hs}}\bigg)^{p+1}{\rm d}x + \circ({\bf 1}) \no\\
	&= \frac{1}{\|w\|_{\Hs}^{p+1}}\int_{\Rn} \frac{x+y_n}{|x+y_n|}|w^*(x)|^{p+1}{\rm d}x\xrightarrow{n\to\infty} e \text{ for some }e\in S^{N-1}.\no
	\end{align}
	Hence we arrive at a contradiction.
\end{proof}
Now we introduce Lusternik Schnirelman (L-S) category theory which will help us to obtain second and third positive solution of \eqref{MAT1}. %For details one can look for {\color{blue}\cite{AM}}.

\begin{definition}
	For a topological space $M$, a nonempty subset $A$ of $M$ is said to be contractible in $M$ if the inclusion map $i:A\hookrightarrow M$ is homotopic to a constant $p\in M$, namely there is map $\eta \in C\big([0,\;1]\times A,\;M\big)$ such that for some $p\in M$
	\begin{align}
	(i)\;\eta(0,\;u) &= u\;\text{ for all } u\in A\notag\\
	(ii)\;\eta(1,\;u) &= p\;\text{ for all } u\in A\notag
	\end{align}
\end{definition}

\begin{definition}
	The (L-S) category of $A$ with respect to $M$, denoted by $cat(A,\;M)$, is the least non-negative integer $k$ such that $A\subset \cup_{i=1}^{k}A_i$, where each $A_i\;(1\leq i\leq k)$ is closed and contractible in $M$. We set $cat(\emptyset,\;M) =0$ and $cat(A,\;M)= +\infty$ if there are no integers with the above property. We write $cat(M)$ to denote $cat(M,\;M)$.
\end{definition}

For fundamental properties of Lusternik-Schnirelman category, we refer
to Ambrosetti \cite{Ambro}. Here we use the following property \cite{Ambro} 
(also see \cite[Proposition~2.4]{Adachi}) which will play a pivotal role to obtain second and third solutions of \eqref{MAT1}.

\begin{proposition} \lab{p:15-8-3}
Suppose $M$ is a Hilbert manifold and $\Psi \in C^{1}(M,\;\R)$. Assume that for $c_0\in\R$ and $k\in\mathbb{N}$,

(i) $\Psi(x)$ satisfies  $(PS)_c$ for $c\leq c_0$,

(ii) $ cat \big(\{x\in M\;:\;\Psi(x)\leq c_0\}\big)\geq k$.

Then $\Psi(x)$ has at least $k$ critical points in $\{x\in M\;:\;\Psi(x)\leq c_0\}$.
\end{proposition}

\begin{lemma}\lab{l:15-8-2} \textup{(\cite[Lemma~2.5]{Adachi})}
	Let $N\geq 1$ and $M$ be a topological space and $S^{N-1}$ denote the unit sphere in $\Rn$. Suppose the there exists two continuous mapping 
	$$F:S^{N-1}\to M,\quad G:M\to S^{N-1},$$
	such that $G\circ F$ is homotopic to the identity map $Id:S^{N-1}\to S^{N-1}$, namely there is continuous map $\eta:[0,\;1]\times S^{N-1}\to S^{N-1}$ such that 
	\begin{align}
	\eta(0,\;x) &= (G\circ F)(x)\;\text{ for all } x\in S^{N-1}\notag\\
	\eta(1,\;x) &= x\;\text{ for all }x\in S^{N-1}.\notag
	\end{align}
	Then $cat(M)\geq 2$.
\end{lemma}
In view of the above lemma, our next goal will be to construct two mappings:
\begin{align}
&F\;:\;S^{N-1}\to \[J_{a,f}\leq I_{a,f}(\ul)+I_{1,0}(w^*)-\varepsilon\],\notag\\
&G\;:\;\[J_{a,f}\leq I_{a,f}(\ul)+I_{1,0}(w^*)-\varepsilon\]\to S^{N-1},\notag
\end{align}
so that $G\circ F$ is homotopic to the identity.

\begin{proposition}\lab{p:12-8-1}
	Let $a$ be as in Theorem \ref{MT} and $d_2>0$ and $R_0>0$ be as found in Proposition \ref{ATP1.7} and Proposition \ref{main-energy-estimate} respectively. Then there exists $d_3\in(0,\;d_2]$ and $R_1>R_0$, such that for any $0<\|f\|_{H^{-s}(\Rn)}\leq d_3$ and for any $|y|\geq R_1$, there exists a unique $t=t(f,y)>0$ in a neighbourhood of $1$ satisfying
\begin{eqnarray}\lab{12-8-4}
\ul+tw^*(x-y) &=& t_{a,f}\bigg(\frac{\ul+tw^*(x-y)}{\|\ul+tw^*(x-y)\|_{\Hs}}\bigg)\\
&&\qquad\times \frac{\ul+tw^*(x-y)}{\|\ul+tw^*(x-y)\|_{\Hs}}. 
\end{eqnarray}	
 Moreover,
	$$\{y\in\Rn \;:\;|y|>R_1\}\to (0,\;\infty); \quad y\mapsto t(f,y)$$
	is continuous.
	Here $w^*$ is the unique ground state solution of \eqref{AT0.8}.
\end{proposition}
\begin{proof}
	 Using implicit function theorem, the proof follows exactly in the same spirit of \cite[Proposition 2.6]{Adachi}. We skip the details.
\end{proof}

\vspace{2mm}

Let us define $F_{R}:S^{N-1}\to\tilde\Sigma_{+}$ in the following way:
$$F_{R}(y) = \frac{\ul+t(f,Ry)w^*(x-Ry)}{\|\ul+t(f,Ry)w^*(x-Ry)\|_{\Hs}},$$
for $\|f\|_{H^{-s}(\Rn)}\leq d_3$ and $R\geq R_1$. 

In Proposition \ref{main-energy-estimate}, we have noticed that for $|y|\geq R_0$, \eqref{main-AT2.2} holds for all $t\geq 0$. For $|y|\geq R_0$, we choose $t=t(f,y)$ such that \eqref{12-8-4} holds.

Therefore,
\Bea J_{a,f}\bigg(\frac{\ul+tw^*(x-y)}{\|\ul+tw^*(x-y)\|_{\Hs}}\bigg) &=& I_{a,f}(\ul+tw^*(x-y))\\
&<&I_{a,f}(\ul)+I_{1,0}(w^*).\Eea

\begin{proposition}\textup{(\cite[Proposition~2.7]{Adachi})}
Let $d_3$ and $R_1$ be as found in Proposition \ref{p:12-8-1}. Then, for $0<\|f\|_{H^{-s}(\Rn)}\leq d_3$ and $R\geq R_1$, there exsits $\varepsilon_{0}=\varepsilon_0(R)>0$ such that
	$$ F_{R}(S^{N-1})\subseteq \[J_{a,f}\leq I_{a,f}(\ul)+I_{1,0}(w^*)-\varepsilon_{0}(R)\],$$
	where the notation $\[J_{a,f}\leq c\]$ is meant in the sense of \eqref{12-8-5}.
\end{proposition}

\begin{proof}
	By construction, we have,
	$$F_{R}(S^{N-1}) \subseteq \[J_{a,f}< I_{a,f}(\ul)+I_{1,0}(w^*)\].$$
	Since, $F(S^{N-1})$ is compact, the conclusion holds.
\end{proof}
Thus we construct a mapping
$$F_R : S^{N-1}\to  \[J_{a,f}\leq I_{a,f}(\ul)+I_{1,0}(w^*)-\varepsilon_{0}(R)\]$$
Now we will construct $G$. For the construction of $G$ the following lemma is important.
\begin{lemma}\lab{l:15-8-1}
	There exists $d_4\in(0,\;d_3]$ such that if $\|f\|_{H^{-s}(\Rn)}\leq d_4$, then
	\be\label{AT2.7}
	\[J_{a,f}< I_{a,f}(\ul)+I_{1,0}(w^*)\] \subseteq \[J_{a,0}<I_{1,0}(w^*)+\delta_{0}\]
	\ee
	where $\delta_{0}>0$ is given in lemma \eqref{ATL1.12}.
\end{lemma}

\begin{proof}
%We prove this lemma in the spirit of \cite[Lemma~2.8]{Adachi}.	
From \eqref{AT1.13}, we have for any $\varepsilon \in (0,1)$
	\be\label{AT2.8}
	J_{a,0}(v) \leq (1-\varepsilon)^{-\tfrac{p+1}{p-1}}\bigg(\Jaf + \frac{1}{2\varepsilon}\|f\|_{H^{-s}(\Rn)}^2\bigg)\;\text{for all } v\in \tilde\Sigma_{+}
	\ee
From \eqref{14-8-1}, we also have $I_{a,f}(\ul)<0$. 

Therefore, if
$v\in \[J_{a,f}< I_{a,f}(\ul) +I_{1,0}(w^*)\]$ then  $\Jaf < I_{1,0}(w^*)$. Consequently, from \eqref{AT2.8}, we have
	$$J_{a,0}(v)\leq (1-\varepsilon)^{-\tfrac{p+1}{p-1}}\bigg(I_{1,0}(w^*)+\frac{1}{2\varepsilon}\|f\|_{H^{-s}(\Rn)}^2\bigg)$$
for all $v\in\[J_{a,\;f}\leq I_{a,f}(\ul)+I_{1,0}(w^*)\]$.
%	Therefore, $v\in \[J_{a,0}\leq (1-\varepsilon)^{-\tfrac{p+1}{p-1}}\bigg(I_{1,0}(w)+\frac{1}{2\varepsilon}\|f\|_{H^{-s}(\Rn)}^2\bigg)\]$\\
	Since $\varepsilon \in (0,1)$ is arbitrary, we have 
	$$v\in \[J_{a,0}<I_{1,0}(w^*)+\delta_{0}\] \text{ for sufficiently small   }\|f\|_{H^{-s}(\Rn)}.$$
Hence the lemma follows.
\end{proof}
Now we can define, $G: \[J_{a,f}<I_{a,f}(\ul)+I_{1;0}(w^*)\]\to S^{N-1}$ by

$$G(v):=\dfrac{\displaystyle\int_{\Rn} \tfrac{x}{|x|}|v|^{p+1}{\rm d}x}{\bigg|\displaystyle\int_{\Rn} \tfrac{x}{|x|}|v|^{p+1}{\rm d}x\bigg|},$$
which is well defined thanks to Lemma \ref{ATL1.12} and Lemma \ref{l:15-8-1}.
Moreover, we will prove that these constructions $F$ and $G$ serves our purpose.
\begin{proposition}\lab{l:15-8-3}
	For a sufficiently large $R\geq R_1$ and for sufficiently small $\|f\|_{H^{-s}(\Rn)}>0$,
	$$G\circ F_R:S^{N-1}\to S^{N-1}$$
	is homotopic to identity.
\end{proposition}
\begin{proof}
This proof follows in the same spirit as in \cite[Proposition~2.4]{Adachi}. We skip the details. 
\end{proof}

We are now in a position to state our main result in this subsection:
\begin{proposition}\lab{p:15-8-1}
	For sufficiently large $R\geq R_1$,
	$$cat\bigg( \[J_{a,f}<I_{a,f}(\ul)+I_{1,\;0}(w^*)-\varepsilon_{0}(R)\] \bigg)\geq 2.$$
\end{proposition}
\begin{proof}
Combining Lemma \ref{l:15-8-2} and Proposition \ref{l:15-8-3}, this proof follows. 
\end{proof}
The above proposition led us to the following multiplicity results.

\begin{theorem}\lab{t:2nd-3rd sol}
	Let $a$ be as in Theorem \ref{MT}. Then there exists $d_5>0$ such that if $\|f\|_{H^{-s}(\Rn)}\leq d_5$ and $f$ is nonnegative nontrivial functional in $H^{-s}(\Rn)$, then $J_{a,f}$ has at least two critical points in
	$$\[J_{a,f}<I_{a,f}(\ul)+I_{1,0}(w^*)\].$$
\end{theorem}
\begin{proof}
From Corollary \ref{corr-ps}, we know $(PS)_c$ is satisfied for $J_{a,f}$ when
$c\in (-\infty, I_{a,f}(\ul)+I_{1,0}(w^*))$. Hence the theorem follows from Proposition \ref{p:15-8-1} and Proposition \ref{p:15-8-3}.
\end{proof}

\medskip

{\bf Proof of Theorem \ref{MT} concluded}: 
\begin{proof}
We set the first positive solution as $u_1:=u_{locmin}(a,f,x)$ which was found in Proposition \ref{ATl1.4}.  Further, \eqref{14-8-1} implies
  $$I_{a,f}(\ul)< 0.$$
By Theorem \ref{t:2nd-3rd sol}, $J_{a,f}$ has at least two critical points $v_2$, $v_3$ in
	$$\[J_{a,f}<I_{a,f}(\ul)+I_{1,0}(w^*)\].$$
Using Proposition \ref{ATP1.7}(iii), $u_2:=t_{a,f}(v_2)v_2$ and $u_3:=t_{a,f}(v_3)v_3$ are the 2nd and 3rd positive solutions of \eqref{MAT1}. Further, by Lemma \ref{ATl1.2}(iii),  $0<J_{a,f}(v_i)=I_{a,f}(u_i)$, $i=1,2$. Hence
$$0<I_{a,f}(u_i)<I_{a,f}(u_1)+I_{1,0}(w^*), \quad i=1,2.$$
%Finally, by Theorem \ref{t:16-8-1}, $J_{a,f}$ has another critical point $v_4$ and we denote $u_4:=t_{a,f}(v_4)v_4$ is the corresponding solution and it satisfies
%$$I_{a,f}(u_4)=J_{a,f}(v_4)\geq I_{a,f}(u_1)+I_{1,0}(w).$$
Hence $u_1, u_2, u_3$ are distinct and \eqref{MAT1} has at least $3$ distinct solutions. 
\end{proof}

\section{Existence Result when $f\equiv 0$}

In this section we aim to prove Theorem \ref{th:1} in the spirit of \cite{DN}. For this using Mountain pass theorem, we first attempt to solve the following problem in the bounded domain with Dirichlet boundary condition:

\begin{equation}\label{P_k}
  \tag{$\mathcal P_k$}
\left\{\begin{aligned}
		(-\Delta)^s u + u &= a(x) |u|^{p-1}u \;\;\text{in}\;B_k,\\
		u &>0 \quad\text{in}\quad B_k,\\
		u &=0 \quad\text{in}\quad \Rn\setminus B_k,
		 \end{aligned}
  \right.
\end{equation}
where $B_k$ denotes the ball of radius $k$, centered at origin, $ 0<a\in L^\infty(\Rn)$ satisfies 
\be\lab{9-10-4} \lim_{|x|\to\infty} a(x) = a_0=\inf_{x \in \Rn} a(x).\ee 

\begin{remark}\lab{r:9-10-1}Without loss of generality, we can assume $a\not\equiv a_0$, since if $a\equiv a_0$ then $u=a_0^{-\frac{1}{p-1}}w^*$ is a solution of \eqref{MAT1} \textup{(with $f\equiv 0$)}, where $w^*$ is the unique ground state solution of \eqref{AT0.8}. In this case Theorem \ref{th:1} follows immediately.
\end{remark}

 We fix some notations first. Denote,
$$E_k:=\{u\in \Hs:  u=0 \quad\text{in}\quad \Rn\setminus B_k\}$$
 i.e., $E_k$ is the closure of $C_0^\infty(B_k)$ w.r.t the norm in $\Hs$. Therefore,
 $$E_1\subseteq E_2\subseteq\cdots\subseteq H^s(\Rn)$$ and
 $\cup_{k\geq 1}E_k$ is dense in $\Hs$.
We define $I_{a,0}$ as in \eqref{EF-1} (taking $f=0$ there) and let $I_{a,0}^k$ denote the restriction of $I_{a,0}$ to the subspace $E_k$. Like before using Remark \ref{r:30-7-3}, we can conclude that any nontrivial critical point of $I_{a,0}^k$ is necessarily a positive solution of \eqref{P_k}.

\begin{lemma}\lab{l:9-8-1}
For each $k\geq 1$, Dirichlet problem \eqref{P_k} admits a solution $u_k$. Moreover, $\{u_k\}$ is uniformly bounded in $\Hs$ and so it contains a subsequence that converges weakly to $\bar u\geq 0$ in $\Hs$.
\end{lemma}
\begin{proof}
For $u\in E_k$,
\Bea
I_{a,0}^k(u)=I_{a,0}(u)&=&\frac{1}{2}\|u\|_{\Hs}^2-\frac{1}{p+1}\int_{\Rn}a(x)u_+^{p+1}\rm{dx}\\
&\geq& \frac{1}{2}\|u\|_{\Hs}^2-\frac{1}{p+1}\|a\|_{L^\infty(\Rn)}\|u\|_{L^{p+1}(\Rn)}^{p+1}\\
&\geq& \frac{1}{2}\|u\|_{\Hs}^2-\frac{1}{p+1}\|a\|_{L^\infty(\Rn)}S_1^{-\frac{p+1}{2}}\|u\|_{\Hs}^{p+1},
\Eea
where $S_1$ is as defined in \eqref{17-7-5}. As $\cup_{k\geq 1} E_k$ is dense in $\Hs$, 
$$I_{a,0}(u)\geq \frac{1}{2}\|u\|_{\Hs}^2-\frac{1}{p+1}\|a\|_{L^\infty(\Rn)}S_1^{-\frac{p+1}{2}}\|u\|_{\Hs}^{p+1} \quad\forall\, u\in\Hs.$$
Therefore, there exist $\de, \, \bar\al>0$ such that
$$I_{a,0}(u)\geq\bar \al>0 \quad\text{on}\quad \|u\|_{\Hs}=\de,  \quad u \in\Hs. $$
Now, choose $u_0\in E_1$ with $\|u_0\|_{\Hs}>\de$ and $u_0\geq 0$. As $p>1$, there exists $t_0>0$ such that
$$I_{a,0}(tu_0)\leq \frac{t^2}{2}\|u\|_{\Hs}^2-\frac{a_0t^{p+1}}{p+1}\int_{B_1}u_0^{p+1} {\rm d}x < 0 \quad\forall\, t>t_0.$$
If $t_0\leq 1$, then we define $e:=u_0$ otherwise, we define $e:=t_0u_0$. Therefore, $0\leq e\in E_1$, $\|e\|\geq \de$ and $I_{a,0}(te)<0$ for all $t>1$.
We define $\Ga_k$ to be the set of all continuous paths in $E_k$ connecting $0$ and $e$ and $\Ga$ be the set of all continuous paths in $\Hs$ connecting $0$ and $e$. Set,
\be\lab{9-8-1}
\al:=\inf_{\ga\in\Ga}\max_{u\in\ga}I_{a,0}(u)
\ee 
and \be\lab{9-8-2}
\al_k:=\inf_{\ga\in\Ga_k}\max_{u\in\ga}I_{a,0}(u).
\ee 
Observe that, $\Ga_1\subseteq\Ga_2\subseteq\cdots\subseteq\Ga$ and this in turn implies
\be\lab{9-8-3}
\al_1\geq\al_2\geq\cdots\geq \al\geq\bar\al>0.
\ee
Moreover, since $\cup_{k\geq 1}E_k$ is dense in $\Hs$, it is easy to check that $\al_k\to\al$ as $k\to\infty$. Applying Mountain pass lemma, we obtain $\al_k$ is a critical point of $I_{a,0}^k$. Let $u_k\in E_k$ be the critical point of $I_{a,0}^k$ corresponding to $\al_k$. Therefore, $I_{a,0}(u_k)=I_{a,0}^k(u_k)=\al_k$ and 
$(I_{a,0}^k)'(u_k)=0$. In particular,
$$\al_k=I_{a,0}^k(u_k)-\frac{1}{p+1}(I_{a,0}^k)'(u_k)u_k=\bigg(\frac{1}{2}-\frac{1}{p+1}\bigg)\|u_k\|^2_{\Hs}.$$
Since $\al_k\leq \al_1$, for all $k\geq 1$, from the above expression we obtain 
$\{u_k\}$ is uniformly bounded in $\Hs.$
%\be\lab{9-8-4}
%0=(I_{a,0}^k)'(u_k)u_k=\|u_k\|_{\Hs}^2-\int_{B_k}a(x)(u_k)_+^{p+1} dx
%\ee
%and 
%\be\lab{9-8-5}
%\al_k=I_{a,0}^k(u_k)=\frac{1}{2}\|u_k\|_{\Hs}^2-\frac{1}{p+1}\int_{B_k}a(x)(u_k)_+^{p+1} dx
%\ee
%Consequently, $(\frac{1}{2}-\frac{1}{p+1})\|u_k\|^2=\al_k\leq \al_1$. 
Hence there exists $\bar u$ in $\Hs$ such that, up to a subsequence, $u_k\rightharpoonup \bar u$ in  $\Hs$. Moreover, $I_{a,0}(u_k)=\al_k\geq \bar \al>0$ implies $u_k$ is nontrivial. Therefore, taking $(u_k)_-$ as the test function in 
\begin{equation}\lab{9-12-1}
%\label{P_k}\tag{$\mathcal P_k$}
\left\{\begin{aligned}
		(-\Delta)^s u + u &= a(x) u^{p}_+ \;\;\text{in}\;B_k,\no\\
		%u &>0 \quad\text{in}\quad B_k,\\
		u &=0 \quad\text{in}\quad \Rn\setminus B_k,
		 \end{aligned}
  \right.
\end{equation}
we obtain $u_k\geq 0$. Therefore, using maximum principle we  have $u_k>0$ in $B_k$. Hence $\bar u\geq 0$.

%First we note that if $\la_{1k}$ denotes the first eigenvalue of $(-\De)^s+I$ in $B_k$ with Dirichlet boundary condition, then $\la_{1k}\geq 1$. 

\end{proof}
\begin{lemma}\lab{l:9-8-2}
Let $u_k$ be a critical point of $I_{a,0}^k$ and $u_k\rightharpoonup \bar u$ in $\Hs$. Then $\bar u$ is a critical point of $I_{a,0}$.
\end{lemma}
\begin{proof}
$u_k$ be a critical point of $I_{a,0}^k$ implies %$(I_{a,0}^k)'(u_k)\psi=0$ for all $\psi\in E_k$. i.e., 
$$
\langle u_k, \psi\rangle_{\Hs}-\int_{B_k}a(x)(u_k)_+^{p} \, \psi\, {\rm d}x = 0 \quad\forall\, \psi\in E_k.
$$
Let $\phi\in C_0^\infty(\Rn)$. Then $\phi\in E_k$ for large $k$. Therefore, for large $k$,
\be\lab{9-10-2}\langle u_k, \phi\rangle_{\Hs}-\int_{\Rn}a(x)(u_k)_+^{p}\phi\, {\rm d}x = 0.\ee
Since $u_k\rightharpoonup \bar u$ in $\Hs$ implies $\langle u_k, \phi\rangle_{\Hs}\to \langle \bar u, \phi\rangle_{\Hs}$ and by lemma \ref{L2},

 $\displaystyle\int_{\Rn}a(x)(u_k)_+^{p} \, \phi\,{\rm d}x \to \int_{\Rn}a(x)\bar u_+^{p} \, \phi\, {\rm d}x.$ Therefore letting $k \rightarrow \infty$ in \eqref{9-10-2}, yields $I'_{a,0}(\bar u)(\phi)=0$. Since $\phi\in C^\infty_0(\Rn)$ is arbitrary, the lemma follows.
\end{proof}

Thanks to Lemma \ref{l:9-8-1} and Lemma \ref{l:9-8-2}, we are just left to show that $\bar u\not\equiv 0$, in order to complete the proof of Theorem \ref{th:1}.

\subsection{Comparison argument}
For any arbitrarily fixed $R>0$, define,
\begin{equation}
\label{eq:h}
h_R(x)=\left\{\begin{aligned}
		a(x)\;\;\text{if}\quad |x|>R,\\
		0 \;\;\text{if}\quad |x|\leq R.
				 \end{aligned}
  \right.
\end{equation}
We define the following Nehari manifolds:
$$\mathcal{N}:= \{u\in\Hs\setminus\{0\}: \|u\|_{\Hs}^2=\int_{\Rn}a(x)u_+^{p+1} \, {\rm d}x\}$$ and
$$\mathcal{N}_{R}:= \{u\in\Hs\setminus\{0\}: \|u\|_{\Hs}^2=\int_{\Rn}h_R(x)u_+^{p+1} \, {\rm d}x\}$$
Set, \be\lab{9-9-1}
\al^*:=\inf_{u\in\mathcal{N}}I_{a,0}(u) \quad\text{and}\quad \ba^*_R:=\inf_{u\in\mathcal{N}_{R}}I_{h_R,0}(u).
\ee
%where $$I(u):=\frac{1}{2}\|u\|^{2}_{H^s(\R^N)}-\frac{1}{p+1}\int_{\R^N}h_R(x)u_+^{p+1}\,dx. $$%\quad\text{and}\quad H(x,u):=\int_{0}^u h(x,t) dt.$$

\begin{lemma}\lab{l:9-10-1}
$\al^*<\lim_{R\to\infty}\ba^*_R$.
\end{lemma}
\begin{proof}
From the definition of $I_{a,0}$, we have
$$\al^*= \bigg(\frac{1}{2}-\frac{1}{p+1}\bigg)\inf_{\mathcal{N}}\int_{\Rn}a(x)u_+^{p+1}{\rm d}x= \bigg(\frac{1}{2}-\frac{1}{p+1}\bigg)\inf_{\Hs\setminus\{0\}}\bigg[\frac{\|u\|_{\Hs}}{\big(\int_{\Rn} a(x)u_+^{p+1}{\rm d}x \big)^\frac{1}{p+1}}\bigg]^\frac{2(p+1)}{p-1}.$$
Similarly, 
$$\ba^*_R=\bigg(\frac{1}{2}-\frac{1}{p+1}\bigg)\inf_{\Hs\setminus\{0\}}\bigg[\frac{\|u\|_{\Hs}}{\big(\int_{\Rn}h_R(x)u_+^{p+1}{\rm d}x\big)^\frac{1}{p+1}}\bigg]^\frac{2(p+1)}{p-1}.$$
Therefore, it is enough to prove 
$$\inf_{\Hs\setminus\{0\}}\frac{\|u\|_{\Hs}}{\big(\int_{\Rn}a(x)u_+^{p+1}{\rm d}x\big)^\frac{1}{p+1}}<\lim_{R\to\infty}\bigg(\inf_{\Hs\setminus\{0\}}\frac{\|u\|_{\Hs}}{\big(\int_{\Rn}h_R(x)u_+^{p+1}{\rm d}x\big)^\frac{1}{p+1}}\bigg).$$

Equivalently, it is enough to show
\be\lab{9-10-1}\sup_{\|u\|_{\Hs}=1}\int_{\Rn}a(x)u_+^{p+1} {\rm d}x>\lim_{R\to\infty}\bigg(\sup_{\|u\|_{\Hs}=1}\int_{|x|>R}a(x)u_+^{p+1} {\rm d}x\bigg).\ee
From \eqref{9-10-4} we have $\lim_{|x|\to\infty} a(x) = a_0=\inf_{x \in \Rn} a(x)$. 
 In view of Remark \ref{r:9-10-1}, we first note that, it is enough to consider the 
case when $\mu(\{x\in\Rn: a(x)\neq a_0\})>0$, where $\mu(X)$ denotes the Lebesgue measure of a set $X$. In this case, we 

{\bf Claim:} \be\lab{9-10-5}\sup_{\|u\|_{\Hs}=1}\int_{\Rn}a(x)u_+^{p+1} {\rm d}x>M:=\sup_{\|u\|_{\Hs}=1}\int_{\Rn}a_0u_+^{p+1} {\rm d}x.\ee

To see the claim, first we note that clearly, for each $u\in\Hs$ with $\|u\|_{\Hs}=1$, we have 
$$\int_{\Rn}a(x)u_+^{p+1} {\rm d}x>\int_{\Rn}a_0u_+^{p+1} {\rm d}x.$$
Therefore, the claim will be proved if we show that $M$ is attained. For that, let $v_n$ be a maximizing sequence, i.e.,
$$\|v_n\|_{\Hs}=1, \quad \int_{\Rn}a_0(v_n)_+^{p+1} {\rm d}x\to M.$$
Using symmetric rearrangement technique, without loss of generality, we can
assume that $v_n$ is radially symmetric and symmetric decreasing (see \cite{Frank-2}). We denote by $H^s_{rad, d}(\Rn)$, the set of all radially symmetric and decreasing functions in $\Hs$. Using \cite[Lemma 6.1]{BM-2}, it is easy to see that 
$$H^s_{rad, d}(\Rn)\hookrightarrow L^{p+1}(\Rn)$$ is compact. Hence by standard argument, it follows that $M$ is attained. Therefore the claim follows.  

Thanks to the above above claim, we have
\bea\lab{10-9-1}
\sup_{\|u\|_{\Hs}=1}\int_{\Rn}a(x)u_+^{p+1} {\rm d}x &>&\sup_{\|u\|_{\Hs}=1}\int_{\Rn}a_0u_+^{p+1} {\rm d}x\no\\
&>& \sup_{\|u\|_{\Hs}=1}\int_{|x|>R}a_0u_+^{p+1} {\rm d}x\no\\
&\geq&\lim_{R\to\infty}\bigg( \sup_{\|u\|_{\Hs}=1}\int_{|x|>R}a_0u_+^{p+1} {\rm d}x\bigg).
\eea
Using the fact that $\lim_{|x|\to\infty} a(x)\to a_0$, a straight forward computation yields
$$\lim_{R\to\infty}\bigg( \sup_{\|u\|_{\Hs}=1}\int_{|x|>R}a_0u_+^{p+1} {\rm d}x\bigg)=\lim_{R\to\infty}\bigg( \sup_{\|u\|_{\Hs}=1}\int_{|x|>R}a(x)u_+^{p+1} {\rm d}x\bigg).$$
Substituting the above equality into \eqref{10-9-1}, we obtain \eqref{9-10-1}.

%On the other hand, given any $\eps>0$, there exists $R>0$ large enough such that
%\Bea\sup_{\|u\|_{\Hs}=1}\int_{|x|>R}a(x)u_+^{p+1} {\rm d}x &\leq& (a_0+\eps)\sup_{\|u\|_{\Hs}=1}\int_{|x|>R}u_+^{p+1} {\rm d}x\\
%%&\leq&\frac{a_0+\eps}{a_0}\sup_{\|u\|_{\Hs}=1}\int_{|x|>r}a_0u_+^{p+1} {\rm d}x\\
%&\leq&\frac{a_0+\eps}{a_0}\sup_{\|u\|_{\Hs}=1}\int_{\Rn}a_0u_+^{p+1} {\rm d}x\\
%&\leq&\frac{a_0+\eps}{a_0}\sup_{\|u\|_{\Hs}=1}\int_{\Rn}a(x)u_+^{p+1} {\rm d}x,
%\Eea
%where for the last inequality we have used \eqref{9-10-5}. Therefore, 
%$$\lim_{R\to\infty}\bigg(\sup_{\|u\|_{\Hs}=1}\int_{|x|>R}a(x)u_+^{p+1} {\rm d}x\bigg)\leq\frac{a_0+\eps}{a_0}\sup_{\|u\|_{\Hs}=1}\int_{\Rn}a(x)u_+^{p+1} {\rm d}x.$$
%As $\eps>0$ is arbitrary, the above expression yields \eqref{9-10-1}. Hence the lemma holds.

\end{proof}

\begin{lemma}\lab{l:9-11-1}
$\al\leq \al^*$, where $\al$ and $\al^*$ are defined as in \eqref{9-8-1} and \eqref{9-9-1}.
\end{lemma}
\begin{proof}
Let $v\in\mathcal{N}$ be arbitrarily chosen and  $V$ denote the $2$-dimensional subspace  spanned by $v$ and $e$, where $e$ is as found in the proof of Lemma \ref{l:9-8-1}. Let $V^+:=\{av+be: a\geq 0, \, b\geq 0\}$. Let $S$ be the circle on $V$ with radius $R$ large enough such that $I_{a,0}\leq 0$ on $S\cap V^+$ (this follows since $p>1$ and standard compactness argument on $V^+$) and $v,\, e$ lie inside $S$. Let $l_{v}:=\{tv: t\geq 0\}$ and $l_{e}:=\{te: t\geq 0\}$ intersect $S$ at $v_1$ and $v_2$ respectively. We define, $\tilde \ga$ be the path that consists of the segment on $l_{v}$ with endpoints $0$ and $v_1$, the arc $S\cap V^+$ (connecting $v_1$ and $v_2$) and the segment on $l_{e}$ with endpoints $v_2$ and $e$. Therefore, clearly 
$\tilde\ga\in\Gamma$ and $v\in\tilde\ga$. 

{\bf Claim:} $\max_{u\in\tilde\ga}I_{a,0}(u)=I_{a,0}(v).$

Indeed, a straight forward computation yields 
$$v\in\mathcal{N}\quad\text{implies}\quad \max_{t\geq 0}I_{a,0}(tv)=I_{a,0}(v).$$ Further, from the construction of $\tilde\ga$ it follows $I_{a,0}\leq 0$ on the rest part of $\tilde \ga$ (since $I_{a,0}\leq 0$ on $S\cap V^+$ and $I_{a,0}(te)<0$ for $t>1$). Hence the claim follows. 

The above claim immediately yields  $$\al\leq \max_{u\in\tilde\ga}I_{a,0}(u)=I_{a,0}(v).$$   
On the other hand, as $v\in\mathcal{N}$ was arbitrarily chosen, we obtain
$$\al\leq \inf_{v\in\mathcal{N}}I_{a,0}(v)=\al^*.$$
\end{proof}

\vspace{2mm}

{\bf Proof of Theorem \ref{th:1}}
\begin{proof}
By Lemma  \ref{l:9-8-2} and Lemma \ref{l:9-8-1}, we know that $\bar u$ is a  nonnegative critical point of $I_{a,0}$. Therefore, it's enough to show that $\bar u\nequiv 0$ in $\Hs$. We prove this by method of contradiction. 
Suppose  $\bar u \equiv 0$ in $\Hs$.
Therefore, using Rellich compactness theorem, $u_k\to 0$ in $L^{p+1}_{loc}(\Rn)$. Hence,
$$0\leq \eps_k:=\int_{B_R}a(x)(u_k)_+^{p+1} {\rm d}x\to 0 \quad\text{as}\quad k\to\infty.$$
It is easy to see that for each $k$, there exists unique $t_{k,R}>0$ such that $t_{k,R}u_k\in \mathcal{N}_{R}$, i.e., 
$$t_{k,R}^2\|u_k\|^2_{\Hs}=t_{k,R}^{p+1}\int_{\Rn}h_R(x)(u_k)_+^{p+1} {\rm d}x.$$ 

{\bf Claim:} $\{t_{k,R}\}_{k=1}^{\infty}$ is a bounded sequence.

To prove the claim, first we note that since $u_k$ is critical point of $I_{a,0}^k$, we have
\Bea
\|u_k\|^2_{\Hs}=\int_{\Rn}a(x)(u_k)_+^{p+1} {\rm d}x=\eps_k+\int_{|x|>R}a(x)(u_k)_+^{p+1} {\rm d}x= \eps_k+\int_{|x|>R}h_R(x)(u_k)_+^{p+1} {\rm d}x.
\Eea
Further, $p>1$ implies there exists $\de>0$ such that $p+1>2+\de$. Therefore, if $t_{k,R}\geq 1$ then combining the above two expressions, we obtain 
$$\eps_kt_{k,R}^2+t_{k,R}^2\int_{|x|>R}h_R(x)(u_k)_+^{p+1} {\rm d}x\geq t_{k,R}^{p+1}\int_{|x|>R}h_R(x)(u_k)_+^{p+1} {\rm d}x\geq t_{k,R}^{2+\de}\int_{|x|>R}h_R(x)(u_k)_+^{p+1} {\rm d}x.$$
Consequently,
\bea\lab{9-10-6} t_{k,R}^2\eps_k\geq (t_{k,R}^{2+\de}-t_{k,R}^2)\int_{|x|>R}h_R(x)(u_k)_+^{p+1} {\rm d}x&=&(t_{k,R}^{2+\de}-t_{k,R}^2)\int_{|x|>R}a(x)(u_k)_+^{p+1} {\rm d}x\no\\
&=&(t_{k,R}^{2+\de}-t_{k,R}^2)(\|u_k\|_{\Hs}^2-\eps_k).
\eea
Further, note that $\eps_k\to 0$ and $I_{a,0}^k(u_k)=\al_k$ implies $$\|u_k\|_{\Hs}^2=2\int_{\Rn}a(x)(u_k)_+^{p+1}{\rm d}x+2\al_k\geq 2\al_k\geq 2\bar\al$$
(the last inequality follows from \eqref{9-8-3}). Therefore from \eqref{9-10-6}, we have $$t_{k,R}^2\eps_k\geq\bar\al(t_{k,R}^{2+\de}-t_{k,R}^2) \quad\text{for large} \ k.$$
As a consequence, $t_{k,R}\to 1$ as $k\to\infty$ (for fixed $R>0$).  Hence the claim holds.

\vspace{2mm}

Using the above claim, we have $t_{k,R}u_k\rightharpoonup 0$ in $\Hs$. Further, as $u_k$ is critical point of $I_{a,0}^k$ implies 
$\|u_k\|^2_{\Hs}=\int_{\Rn}a(x)(u_k)_+^{p+1} {\rm d}x$, a straight forward computation yields that 
$\max_{t\geq 0} I_{a,0}(tu_k)=I_{a,0}(u_k)$. Therefore,
\Bea
\al_k=I_{a,0}(u_k)&\geq& I_{a,0}(t_{k,R}u_k)\\
&=&\frac{t_{k,R}^2}{2}\|u_k\|^2_{\Hs}-\frac{t_{k,R}^{p+1}}{p+1}\int_{|x|>R}a(x)(u_k)_+^{p+1}{\rm d}x-\frac{t_{k,R}^{p+1}}{p+1}\int_{|x|<R}a(x)(u_k)_+^{p+1}{\rm d}x\\
&=&I_{h_R,0}(t_{k,R}u_k)-\frac{t_{k,R}^{p+1}}{p+1}\int_{|x|<R}a(x)(u_k)_+^{p+1}{\rm d}x\\
&\geq&\ba_R^*-\frac{t_{k,R}^{p+1}}{p+1}\int_{|x|<R}a(x)(u_k)_+^{p+1}{\rm d}x,
\Eea
where in the last inequality we have used the fact that $t_{k,R}u_k\in\mathcal{N}_R$. As before, 

$\frac{t_{k,R}^{p+1}}{p+1}\displaystyle\int_{|x|<R}a(x)(u_k)_+^{p+1}{\rm d}x\to 0$ as $k\to\infty$ (keeping $R>0$ fixed) . Thus,  taking the limit $k\to\infty$ yields $\al\geq \ba_R^*$, where $\al$ is as defined in \eqref{9-8-1}. Consequently, $\al\geq \lim_{R\to\infty}\ba_R^*.$
Combining  this with  Lemma \ref{l:9-11-1}, we obtain $\lim_{R\to\infty}\ba_R^*\leq \al^*$. This contradicts Lemma \ref{l:9-10-1}. Hence $\bar u\neq 0$. Therefore,  $\bar u$ is a nontrivial nonnegative critical point of $I_{a,0}$. Finally, thanks to maximum principle \cite[Theorem 1.2]{DPQ}, we get $\bar u$ is a positive solution to \eqref{MAT2} (with $f\equiv 0$). Hence, $\bar u$ is a positive solution to \eqref{MAT1}(with $f\equiv 0$). This completes the proof.  

\end{proof}

\begin{remark}
It is easy to note that if $a(x)\to 0$ as $|x|\to\infty$
at infinity, once again some ``compactness'' exists and standard variational arguments leads to the existence of positive solutions in this case. 
\end{remark}

\begin{remark}
If  $s>\frac{1}{2}$ and $a:\Rn\to [0,\infty)$ is radial function satisfying the growth condition
$$a(r)\leq C(1+r^l) \quad r\geq 0,$$
C>0 being a constant and $l<(N-1)(p-1)/2$, then  proceeding in the spirit of \cite[Lemma 4.8]{DN}, it follows that \eqref{9-12-1} admits a positive radial solution $w_k$ and $\|w_k\|_{\Hs}$ is uniformly bounded above. Therefore, up to a subsequence $w_k\rightharpoonup w$ in $\Hs$ and $w\geq 0$. Next, using the 
radial lemma \cite[Theorem 7.4(i)]{Se}, it can be shown in the similar way as in \cite[Corollary~4.8]{DN} that $w\not\equiv 0$ in $\Hs$ i.e., \eqref{MAT1} (with $f\equiv 0$) admits a positive radial solution in $\Hs$. In this case, we do not need to assume any asymptotic behavior of $a$ at infinity.
\end{remark}

\medskip

{\bf Acknowledgement}: The research of M.~Bhakta is partially supported by the SERB MATRICS grant (MTR/2017/000168). S.~Chakraborty is supported by NBHM grant 0203/11/2017/RD-II and D.~Ganguly is partially supported by INSPIRE faculty fellowship (IFA17-MA98).

\end{document}